\documentclass[10pt,reqno]{amsart}
\usepackage{latexsym,amsmath}
\usepackage{amsfonts}
\usepackage{amssymb}
\usepackage{latexsym}

\usepackage[T1]{fontenc}
\usepackage{fourier}
\usepackage{bbm}
\usepackage{enumerate}
\usepackage{hyperref}
\DeclareMathAlphabet{\mathcal}{OMS}{cmsy}{m}{n}

\newcommand{\beq}{\begin{equation}} \newcommand{\eeq}{\end{equation}}

\newcommand\G{\vec G}

\newcommand\good{sep}

\numberwithin{equation}{section}

\def\vec#1{\mathchoice{\mbox{\boldmath$\displaystyle#1$}}
{\mbox{\boldmath$\textstyle#1$}}
{\mbox{\boldmath$\scriptstyle#1$}}
{\mbox{\boldmath$\scriptscriptstyle#1$}}}

\newcommand{\Zbg}[1]{Z_{\beta}(#1)}
\newcommand{\Zkg}[1]{Z_{\beta,\mathrm{\good}}(#1)}

\newcommand{\Zkb}[1]{Z_{\beta,\mathrm{bal}}(#1)}
\newcommand{\Zkn}[1]{Z_{\beta,\mathrm{\good}}(#1)}

\newcommand{\Zrb}{Z_{\rho,\mathrm{bal}}}

\newcommand{\Dg}{\mathcal D_{\mathrm{\good}}}
\newcommand{\Dng}[1]{\mathcal D_{#1,\mathrm{\good}}}
\newcommand{\Dsg}{\pazocal  D_{s,\mathrm{\good}}}

\newcommand{\rhos}{\rho_{\mathrm{stable}}}

\newcommand{\dc}{d_{k,\mathrm{cond}}}

\newcommand{\dk}{d_{k-\mathrm{col}}}

\DeclareMathOperator{\pr}{\mathbb P}

\newcommand\SIGMA{\vec\sigma}

\newcommand\TAU{\vec\tau}

\newtheorem{definition}{Definition}[section]
\newtheorem{claim}[definition]{Claim}

\newtheorem{theorem}[definition]{Theorem}
\newtheorem{lemma}[definition]{Lemma}
\newtheorem{proposition}[definition]{Proposition}
\newtheorem{corollary}[definition]{Corollary}

\newtheorem{fact}[definition]{Fact}

\usepackage[dvips]{epsfig}
\graphicspath{{./Fig_eps/}{./Eps/}}
\DeclareGraphicsExtensions{.ps,.eps}
\usepackage[usenames,dvipsnames]{xcolor}
\usepackage{fullpage}

\newcommand\sign{\mathrm{sign}}

\newcommand\id{\mathrm{id}}

\newcommand\Forb{\cH_{K_n}}

\newcommand\cA{\mathcal{A}}
\newcommand\cB{\mathcal{B}}

\newcommand\cD{\mathcal{D}}

\newcommand\cH{\mathcal{H}}
\newcommand\cS{\mathcal{S}}

\DeclareMathAlphabet{\pazocal}{OMS}{zplm}{m}{n}
\def\cR{{\mathcal R}}
\def\cRb{{\pazocal R}}
\def\cDb{{\pazocal D}}
\def\cBd{{\pazocal B}}
\def\cSd{{\pazocal S}}

\newcommand\eps{\varepsilon}

\newcommand\Erw{\mathbb{E}}

\newcommand{\vecone}{\vec{1}}

\newcommand{\Bin}{{\rm Bin}}

\newcommand{\bink}[2] {{{#1}\choose {#2}}}

\newcommand\bc[1]{\left({#1}\right)}
\newcommand\cbc[1]{\left\{{#1}\right\}}
\newcommand\bcfr[2]{\bc{\frac{#1}{#2}}}

\newcommand\brk[1]{\left\lbrack{#1}\right\rbrack}

\newcommand\abs[1]{\left|{#1}\right|}

\newcommand\RR{\mathbb{R}}

\newcommand{\whp}{w.h.p.}

\newcommand{\Erdos}{Erd\H{o}s}
\newcommand{\Renyi}{R\'enyi}

\newcommand\Lem{Lemma}
\newcommand\Prop{Proposition}
\newcommand\Thm{Theorem}
\newcommand\Cor{Corollary}
\newcommand\Sec{Section}

\begin{document}

\title{On the Potts antiferromagnet on random graphs}

\author{Amin Coja-Oghlan$^*$ and Nor Jaafari}
\thanks{$^*$The research leading to these results has received funding from the European Research Council under the European Union's Seventh 
	Framework Programme (FP/2007-2013) / ERC Grant Agreement n.\ 278857--PTCC}

\address{Amin Coja-Oghlan, {\tt acoghlan@math.uni-frankfurt.de}, Goethe University, Mathematics Institute, 10 Robert Mayer St, Frankfurt 60325, Germany.}

\address{Nor Jaafari, {\tt jaafari@math.uni-frankfurt.de}, Goethe University, Mathematics Institute, 10 Robert Mayer St, Frankfurt 60325, Germany.}
\date{\today}

\maketitle
\begin{abstract}
	\noindent
Extending a prior result of Contucci et al.\ [Comm.\ Math.\ Phys.\ 2013],
we determine the free energy of the Potts antiferromagnet on the \Erdos-\Renyi\ random graph at all temperatures for average degrees
	$d\le (2k-1)\ln k - 2 - k^{-1/2}$.
In particular, we show that for this regime of $d$ there does not occur a phase transition.
	
	\bigskip
	\noindent
	%\emph{Key words:}	---.
\end{abstract}
	
\section{Introduction}
\subsection{Background and motivation}
The Gibbs measure of the {\em $k$-spin Potts antiferromagnet at inverse temperature $\beta\geq0$}
on a graph $G=(V,E)$ is the probability measure on the set of all maps $\sigma:V\to[k]=\{1,\ldots,k\}$ defined by
\begin{align}\label{eqPotts}
\mu_{G,\beta}(\sigma)&=\frac{\exp(-\beta\cH_G(\sigma))}{Z_\beta(G)},\quad\mbox{where}\quad
\cH_G(\sigma)=\abs{\cbc{e\in E:|\sigma(e)|=1}}\quad\mbox{ and }\quad
\Zbg{G}=\sum_{\tau:V\to[k]}\exp(-\beta\cH_G(\tau)).
\end{align}
Thus, if we think of $[k]$ as a set of colors, then the function $\cH_G$, the {\em Hamiltonian} of $G$, maps a color assignment $\sigma$
to the number of monochromatic edges.
Moreover, $\beta\in[0,\infty)\mapsto \Zbg{G}$ is known as the {\em partition function}.
The Potts antiferromagnet is one of the best-known models of statistical physics.
Accordingly, it has been studied extensively on a wide class of graphs, particularly lattices~\cite{Baxter,MMtrees,LFglass}.
The aim of the present paper is to study the model on the \Erdos-\Renyi\ random graph $\G=\G(n,m)$.
Throughout the paper, we let $m=\lceil dn/2\rceil$ for a number $d>0$ that remains fixed as $n\to\infty$.
We also assume that the number $k\geq3$ of colors remains fixed as $n\to\infty$.

The Potts model on the random graph $\G$ is of interest partly due to the connection to the $k$-colorability problem.
Indeed, the larger $\beta$, the more severe the ``penalty factor'' of $\exp(-\beta)$ that each monochromatic edge induces in (\ref{eqPotts}).
Thus, if the underlying graph is $k$-colorable, then for large $\beta$ the Gibbs measure will put most of its weight on color assignments
	that leave few edges monochromatic.
Ultimately, one could think of the uniform distribution on $k$-colorings as the ``$\beta=\infty$''-case of the Gibbs measure (\ref{eqPotts}).
Now, consider the problem of finding a $k$-coloring of the random graph by a local search algorithm such as Simulated Annealing.
Then most likely the algorithm will start from a color assignment that has quite a few monochromatic edges.
As the algorithm proceeds, it will attempt to gradually reduce the number of monochromatic edges by running the
Metropolis process for the Gibbs measure (\ref{eqPotts}) with a value of $\beta$ that increases over time.
Specifically, $\beta$ has to be large enough to make progress but small enough so that the algorithm does not get trapped in a local minimum of the Hamiltonian.
Hence, to figure out whether such a local search algorithm will find a proper $k$-coloring in polynomial time,
it is instrumental to study the ``shape'' of the Hamiltonian.

To this end, it is key to get a handle on the {\em free energy}, defined as $\Erw[\ln \Zbg{\G}]$.
We take the logarithm because $\Zbg{\G}$ scales exponentially in the number $n$ of vertices.
As a standard application of Azuma's inequality shows that $\ln \Zbg{\G}$ is concentrated about its expectation
(see Fact~\ref{lem:logconc} below),
$\frac1n|\ln \Zbg{\G}-\Erw[\ln \Zbg{\G}]|$ converges to $0$ in probability.
Furthermore, if $\Erw[\ln \Zbg{\G}]\sim\ln \Erw[\Zbg{\G}]$ for certain $d,\beta$,
then the Hamiltonian can be studied via an easily accessible probability distribution called the {\em planted model}.
This trick has been applied to the ``proper'' graph coloring problem as well as to other random constraint satisfaction problems successfully
\cite{Barriers,Molloy}.

\subsection{The main result}
Because our motivation largely comes from the random graph coloring problem, we are going to confine ourselves
to values of $d$ where the random graph $\G$ is $k$-colorable \whp\
Although the precise $k$-colorability threshold $\dk$ is not currently known, we have~\cite{Covers,Danny}
\begin{align}\label{eqkcol}
(2k-1)\ln k-2\ln2+o_k(1)\leq\dk\leq(2k-1)\ln k-1+o_k(1),
\end{align}
where $o_k(1)$ hides a term that tends to $0$ in the limit of large $k$.
The following theorem determines $\frac1n\Erw[\ln\Zbg{\G}]$ almost up to the lower bound from (\ref{eqkcol}).

\begin{theorem}\label{Thm_main}
	There is $k_0>0$ such that for all $k\geq k_0$, $d\le d_{\star}=(2k-1)\ln k-2-k^{-1/2}$, $\beta>0$ we have
	\begin{align}\label{eqThm_main}
	\lim_{n\to\infty}\frac1n\Erw[\ln \Zbg{\G}]&=\lim_{n\to\infty}\frac1n\ln \Erw[\Zbg{\G}]=\ln k+\frac d2\ln(1-(1-\exp(-\beta))/k).
	\end{align}
\end{theorem}

Clearly, the function on the r.h.s.\ of (\ref{eqThm_main}) is analytic in $\beta\in(0,\infty)$.
Thus, in the language of mathematical physics \Thm~\ref{Thm_main} implies that the Potts antiferromagnet
on the random graph does not exhibit a phase transition for any average degree $d<d_{\star}$.

\subsection{Related work}
The problem of determining the $k$-colorability threshold of the random graph was raised in the seminal paper by \Erdos\ and \Renyi\ and is thus the
longest-standing open problem in the theory of random graphs~\cite{ER}.
Achlioptas and Friedgut~\cite{AchFried} proved the existence of a non-uniform sharp threshold.
Moreover, a simple greedy algorithm finds a $k$-coloring for degrees up to about $k\ln k$, approximately half the $k$-colorability threshold~\cite{AchMolloy}.
Further, Achlioptas and Naor~\cite{AchNaor} used the second moment method to establish a lower bound of $\dk\geq2(k-1)\ln k+o_k(1)$,
which matches the first-moment upper bound $\dk\leq(2k-1)\ln k+o_k(1)$ up to about an additive $\ln k$.
Coja-Oghlan and Vilenchik~\cite{Danny} improved the lower bound to $\dk\geq(2k-1)\ln k-2\ln 2+o_k(1)$ via a second moment argument
that incorporates insights from non-rigorous physics work~\cite{LenkaFlorent}.
On the other hand, Coja-Oghlan~\cite{Covers} proved $\dk\leq(2k-1)\ln k-1+o_k(1)$.
The results from~\cite{AchNaor,Danny} were subsequently generalized to various other models, including
random regular graphs and random hypergraphs~\cite{PeterCatherine,Reg,DFG,KPGW}. 

The Potts antiferromagnet on the random graph was studied before by Contucci, Dommers, Giardina and Starr~\cite{CDGS},
who generalized the second moment argument from~\cite{AchNaor} to the Potts model.
In particular, \cite{CDGS} shows that (\ref{eqThm_main}) holds for all $\beta\geq0$ if $d\leq(2k-2)\ln k-2$.
An analogous result was recently obtained  (among other things) by Banks and Moore~\cite{Banks} for a variant of the stochastic block model that resembles
 the Potts antiferromagnet.
Their proof is based on~\cite{AchNaor} as well.
In the present paper we improve the corresponding results of~\cite{Banks,CDGS}
by extending the physics-enhanced second moment argument from~\cite{Danny} to the Potts antiferromagnet.

Physics considerations suggest that for average degrees $d>(2k-1)\ln k-2\ln2+o_k(1)$
a phase transition does occur, i.e., the function $\beta\in(0,\infty)\mapsto\lim_{n\to\infty}\frac1n\Erw[\ln \Zbg{\G}]$ is
non-analytic~\cite{pnas,MM,LenkaFlorent}.
The existence and location of the condensation phase transition has been established asymptotically in the hypergraph $2$-coloring and the hardcore model 
and precisely in the regular $k$-SAT model and the $k$-colorability problem~\cite{Victor,Cond,h2c,BST}.
However, the Potts antiferromagnet is conceptually more challenging than hardcore, $k$-SAT or hypergraph $2$-coloring
because the ``variables'' (viz.\ vertices) can take more than two values (colors).
Potts is also more difficult than $k$-coloring because of the presence of the inverse temperature parameter $\beta$.
In fact, the present work is partly motivated by studying condensation in the Potts antiferromagnet, and
we hope that \Thm~\ref{Thm_main} and its proof may pave the way to pinpointing the phase transition precisely,
	see \Sec~\ref{Sec_cond} below.
Additionally, as mentioned above, \Thm~\ref{Thm_main} implies that for $d\leq(2k-1)\ln k-2-k^{-1/2}$ the Hamiltonian
	can be studied by way of the planted model.
Finally, the {\em ferromagnetic} Potts model (where the Gibbs measure favors monochromatic edges) is far better understood than
the antiferromagnetic version~\cite{DMSS}.

\subsection{Preliminaries}
Throughout the paper we assume that $k\geq k_0$ for a large enough constant $k_0>0$.
Moreover, let
	$$c_\beta=1-\exp(-\beta).$$
Unless specified otherwise, the standard $O$-notation refers to the limit $n\to\infty$. 
We always assume tacitly that $n$ is sufficiently large.
Additionally, we use asymptotic notation in the limit of large $k$ with a subscript  $k$.

\begin{fact}
	\label{lem:logconc}
	For any $\delta>0$ there is $\eps=\eps(\delta,\beta,d)>0$ such that
	$\limsup_{n\to\infty}\frac1n\ln\pr[|\ln \Zbg{\G}-\Erw[\ln \Zbg{\G}]|>\delta n]<-\eps.$
\end{fact}
\begin{proof}
	If $G,G'$ are multi-graphs such that $G'$ can be obtained from $G$ by adding or deleting a single edge, then
	$|\ln\Zbg{G} -\ln\Zbg{G'}|\le 2\beta$.
	Hence, the assertion follows from Azuma's inequality.
\end{proof}

If $s$ is an integer, we write $[s]$ for the set $\{1,\ldots,s\}$.
Further, if $v$ is a vertex of a graph $G$, then $\partial v=\partial_G(v)$ is the set of neighbors of $v$ in $G$.
If $\rho$ is a matrix, then by $\rho_i$ we denote the $i$th row of $\rho$ and by $\rho_{ij}$ the $j$th entry of $\rho_i$.
Further, the Frobenius norm of a $k\times k$-matrix $\rho$ is
	\[ \|\rho\|_2=\brk{\sum_{i,j\in[k]}\rho_{ij}^2}^{1/2}. \]

For a probability distribution $p:\Omega\to[0,1]$ on a finite set $\Omega$ we denote by
	$$H(p)=-\sum_{x\in \Omega}p(x)\ln p(x)$$
the entropy of $p$ (with the convention that $0\ln0=0$). 
Additionally, if $\rho$ is a $k\times k$-matrix with non-negative entries, then we let
	$$H(\rho)=-\sum_{i,j\in[k]}\rho_{ij}\ln\rho_{ij}.$$
%We are going to use the following standard fact about the entropy (cf.~\cite{Danny}).
Further,  $h:[0,1]\to\RR$ denotes the function
	\begin{align*}
	h(z)=-z \ln z -(1-z)\ln(1-z).
	\end{align*}
We will use the following standard fact about the entropy.

\begin{fact}
	\label{fact:entropy}
	Let $p\in[0,1]^k$ be such that $\sum_{i=1}^{k}p_i=1$. 
	Let $I\subset[k]$ and suppose that $q=\sum_{i\in I}p_i\in(0,1)$.
	Then 
		$$H(p)\le h(q)+q\ln |I|+	(1-q)\ln(k-|I|).$$
\end{fact}

\noindent
	\begin{lemma}[Chernoff bound, e.g.\ {\cite{janson2011random}}]
		\label{lem:chernoff} 
		Let $X$ be a binomial random variable with mean $\mu>0$. Then for any $t>1$, we have 
		$
		\pr[X>t\mu]\le \exp[-t\mu \ln(t/e)].$
	\end{lemma}

\section{Outline}
\label{sec:outline}

\noindent
We prove \Thm~\ref{Thm_main} by generalizing the second moment argument for $k$-colorings from~\cite{Danny} to the partition function
of the Potts antiferromagnet.
In this section we describe the proof strategy.
Most of the technical details are left to the subsequent sections.

\subsection{The first moment}
As a first step we calculate the first moment $\Erw[\Zbg{\G}]$.
This is pretty straightforward; in fact, it has been done before~\cite{CDGS}.
Nonetheless, we go over the calculations to introduce a few concepts that will prove important in the second moment argument as well.

\begin{proposition}[\cite{CDGS}]
	\label{prop:firstmoment}
For all $\beta,d>0$ we have
	$\Erw[\Zbg{\G}]=\Theta\left(k^n\left(1-{c_\beta}/{k}\right)^m\right).$
\end{proposition}

To lower-bound $\Zbg{\G}$ we follow Achlioptas and Naor \cite{AchNaor} and work with ``balanced'' color assignments whose color classes
are all about the same size.
Specifically, call $\sigma:[n]\to[k]$ \emph{balanced} if  $\left||\sigma^{-1}(i)|-\frac{n}{k}\right|\le \sqrt{n}$ for all $i=1,\ldots,k$.
Of course, by Stirling's formula the set $\cBd=\cBd(n,k)$ of all balanced $\sigma:[n]\to[k]$ has size $|\cBd|=\Theta(k^n)$.
Let
	\begin{align*}
	\Zkb{\G} = \sum_{\sigma \in \cBd} \exp\left(-\beta \cH_{\G}(\sigma)\right)
	\end{align*}
be the partition function restricted to balanced maps.
Moreover, let
	\begin{align*}
	\Forb(\sigma)=\sum_{i=1}^k \binom{|\sigma^{-1}(i)|}{2}.
	\end{align*}
be the number of monochromatic edges of the complete graph.
Then uniformly for all balanced $\sigma$,
	\begin{align}\label{eqForb0}
	\Forb(\sigma)&=k\binom{\frac{n}{k}+O(\sqrt{n})}{2} 
	=\binom{n}{2}\frac{1}{k}+O(n).
	\end{align}
Hence, by Stirling's formula
	\begin{align}
	\Erw \left[\exp\left(-\beta \cH_{\G}(\sigma)\right)\right]  &
		=\sum_{m_1=0}^m
			\exp(-\beta m_1)\bink{\Forb(\sigma)}{m_1}\bink{\bink{n}2-\Forb(\sigma)}{m-m_1}\bink{\bink{n}2}{m}^{-1}\nonumber\\
		&=\Theta(1)\sum_{m_1=0}^m\bink{m}{m_1}\bcfr{\Forb(\sigma)}{\exp(\beta)\bink n2}^{m_1}
			\bc{1-\frac{\Forb(\sigma)}{\bink n2}}^{m-m_1}.
				\label{eqForb1}
	\end{align}
Combining (\ref{eqForb0}) and (\ref{eqForb1}), we find
	\begin{align}
	\Erw[\Zkb{\G}]=\sum_{\sigma\in\cB}\Erw \left[\exp\left(-\beta \cH_{\G}(\sigma)\right)\right]=
		\Theta\left(k^n\left(1-{c_\beta }/{k}\right)^\frac{nd}{2}\right).\label{eq:balO}
	\end{align}
On the other hand, for all $\sigma$ we have $\Forb(\sigma)\geq \frac{1}{k}\binom{n}{2}-n$ by convexity.
Therefore, (\ref{eqForb1}) yields
	\begin{align}\label{eq:balO_2}
	\Erw[\Zbg{\G}] &= \sum_{\sigma}\Erw \left[\exp\left(-\beta \cH_{\G}(\sigma)\right)\right]\le O\left(k^n\left(1-{c_\beta}/{k}\right)^\frac{nd}{2}\right).
	\end{align}
Combining (\ref{eq:balO}) and (\ref{eq:balO_2}), we obtain \Prop~\ref{prop:firstmoment}.
Moreover, comparing (\ref{eq:balO}) and (\ref{eq:balO_2}), we see that  $\Erw[\Zkb{\G}]$ and $\Erw[\Zbg{\G}]$ are of the same order of magnitude.
Since it is technically more convenient to work with $\Zkb{\G}$, we are going to perform the second moment argument for that random variable.

\subsection{The second moment}
\label{subsec:outline2mm}
Following~\cite{AchNaor}, we define the {\em overlap matrix} 
	$\rho(\sigma,\tau)=(\rho_{ij}(\sigma,\tau))_{i,j\in[k]}$
	of $\sigma,\tau:[n]\to[k]$ by letting
\begin{align}
\label{eq:overlapsdef}
\rho_{ij}(\sigma,\tau)= \frac{k}{n}|\sigma^{-1}(i)\cap \tau^{-1}(j)|.
\end{align}
Thus, $k^{-1}\rho_{ij}(\sigma,\tau)$ is the fraction of vertices with color $i$ under $\sigma$ and color $j$ under $\tau$. 
Let $\cRb=\cRb(n,k)=\{\rho(\sigma,\tau):\sigma,\tau\in\cB\}$ be the set of all possible overlap matrices and set
	\begin{align*}
	\Zrb(\G)=\sum_{\substack{(\sigma,\tau)\in\cB^2\\ \rho(\sigma,\tau)=\rho}}\exp\left(-\beta\left(\cH_{\G}(\sigma)+\cH_{\G}(\tau)\right)\right).
	\end{align*}
Then
	\begin{align}
	\Erw[\Zkb{\G}^2]=\sum_{(\sigma,\tau)\in\cBd^2}\Erw\left[ \exp\left(-\beta\left(\cH_{\G}(\sigma)+\cH_{\G}(\tau)\right)\right)\right] = \sum_{\rho\in\cRb}\Erw [\Zrb(\G)].\label{eq:balancedpairs}
	\end{align}
Further, define
\begin{align}
f_{d,\beta}(\rho)&=
H(k^{-1}\rho)+\frac{d}{2}\ln\left[1-\frac{2}{k}c_\beta+\frac{\|\rho\|_2^2}{k^2}c_\beta^ 2\right].\label{eq:optfunc}
\end{align}
Then an elementary argument similar to the proof of \Prop~\ref{prop:firstmoment} yields

\begin{proposition}[{\cite{CDGS}}]
	\label{prop:exppairs}
	Uniformly for all $\rho\in \cRb$ we have $\Erw [\Zrb(\G)] = \exp(nf_{d,\beta}(\rho)+o(n))$.
\end{proposition}

\noindent
The function $f_{d,\beta}$ is a sum of an entropy term
	$H(k^{-1}\rho)$ and an ``energy term'' 
	\begin{align*}
	E(\rho)=E_{d,\beta}(\rho)=\frac{d}{2}\ln\left[1-\frac{2}{k}c_\beta+\frac{\|\rho\|_2^2}{k^2}c_\beta^2 \right].
	\end{align*}
For future reference we note that
	\begin{align}
	\frac{\partial}{\partial\rho_{ij}}H(k^{-1}\rho)&=\frac1k\left(-1-\ln(\rho_{ij})\right)\label{eq:derivH},\\
	\frac{\partial}{\partial\rho_{ij}}E(\rho)&=\frac{d}{k^2}\frac{c_\beta^2\rho_{ij}}{1-\frac{2}{k}c_\beta+\|\rho\|_2^2c_\beta/k^2}.\label{eq:derivE}
	\end{align}
The number $|\cR|$ of summands on the right hand side of \eqref{eq:balancedpairs} is easily bounded by $n^{k^2}$.
Therefore,
\begin{align}
\label{eq:2mmopt}
\frac1n \ln \Erw[\Zkb{\G}^2]=\frac1n \ln \sum_{\rho\in\cRb}\Erw[\Zrb(\G)]\sim \max_{\rho\in\cRb}\frac{1}{n}\ln \Erw[\Zrb(\G)]\sim \max_{\rho\in\cRb} f_{d,\beta}(\rho).
\end{align}
Denote by $\cS$ the set of all singly-stochastic matrices and by $\cD$ the set of all doubly-stochastic $k\times k$ matrices, respectively.
Then $\bigcup_{n\geq1}\cRb(n,k)\cap\cDb$ is a dense subset of $\cDb$.
Together with (\ref{eq:2mmopt}) the continuity of $f$ therefore implies
\begin{align}
\label{eq:2mmoptbal}
\frac1n \ln \Erw[\Zkb{\G}^2]\sim \max_{\rho\in\cDb}f_{d,\beta}(\rho).
\end{align}
Setting $\bar{\rho}=k^{-1}\vecone$ to be the barycenter of $\cDb$, we obtain from Proposition \ref{prop:exppairs}  that
\begin{align} 
f_{d,\beta}(\bar\rho)\sim \frac2n\ln \Erw \left[\Zkb{\G}\right].
\end{align}
Hence,
just as in the case of proper $k$-colorings~\cite{AchNaor,CDGS},
 a {\em necessary} condition for the success of the second moment method is that the function $f_{d,\beta}$ attains its maximum on $\cD$
at the point $\bar\rho$.

\subsection{Small average degree or high temperature}
Contucci, Dommers, Giardina and Starr~\cite{CDGS} proved that the maximum in (\ref{eq:2mmoptbal}) is indeed attained at $\bar\rho$
if the average degree is a fair bit below the $k$-colorability threshold.

\begin{theorem}[\cite{CDGS}]\label{Thm_CDGS}
Assume that $d<2(k-1)\ln(k-1)$.
Then (\ref{eqThm_main}) holds for all $\beta>0$.
\end{theorem}

Comparing this result with~(\ref{eqkcol}), we see that \Thm~\ref{Thm_CDGS} applies to degrees about an additive $\ln k$
below the $k$-colorability threshold.
The proof of \Thm~\ref{Thm_CDGS} builds upon ideas of Achlioptas and Naor~\cite{AchNaor}.
More precisely, solving the maximization problem from (\ref{eq:2mmoptbal}) directly emerges to be surprisingly difficult.
Hence, Achlioptas and Naor suggested to enlarge the domain to the set of {\em singly} stochastic matrices.
Clearly, the maximum over the larger space is an upper bound on the maximum over the set of doubly-stochastic matrices.
Further, because the set of singly-stochastic matrices is a product of simplices, the relaxed optimization problem can be tackled with a fair bit of
technical work.
Crucially, for $d<2(k-1)\ln(k-1)$ the maximum of the relaxed problem is attained at $\bar\rho$.
However, for only slightly larger values of $d$ the maximum is attained at a different point, and thus the relaxed second moment argument fails.

Apart from the case of small $d$, the second case that is relatively straightforward is that of small $\beta$
	(the ``high temperature'' case in physics jargon).
More precisely, in \Sec~\ref{sec:singlystoch} we will prove the following.

\begin{proposition}
	\label{prop:maxlnkh0}
	If $d\in[ 2(k-1)\ln(k-1),(2k-1)\ln k-2]$ and $\beta\le \ln k$, then (\ref{eqThm_main}) holds.
\end{proposition}

For $d\in[ 2(k-1)\ln(k-1),(2k-1)\ln k-2]$ \Prop~\ref{prop:maxlnkh0} improves upon the result from~\cite{CDGS},
which yields (\ref{eqThm_main}) merely for $\beta\leq\beta_0$ for an absolute constant $\beta_0$ (independent of $k$).
The proof of \Prop~\ref{prop:maxlnkh0} is by way of relaxing (\ref{eq:2mmopt}) to singly-stochastic matrices as well
and builds upon arguments developed in~\cite{Danny} for $k$-colorability.

\subsection{Large degree and low temperature}
\label{subsec:highdeglowtemp}
The most challenging constellation is that of $d$ beyond $2(k-1)\ln(k-1)$ and $\beta$ large.
In this regime we do not know how to solve the maximization problem (\ref{eq:2mmopt}).
In particular, the trick of relaxing the problem to the set of all singly-stochastic matrices does not work.
Instead, following~\cite{Danny} we are going {\em add} further constraints to the problem.
That is, we are going to apply the second moment method to a modified random variable that is constructed
so as to ensure that certain parts of the domain $\cD$ cannot contribute to (\ref{eq:2mmopt}) significantly.

The construction is guided by the physics prediction~\cite{pnas} that for large $d$ and $\beta$ the Gibbs measure $\mu_{\G}$
``decomposes'' into an exponential number of well-separated clusters.
Of course, it would be non-trivial to turn this notion into a precise mathematical statement because the support of $\mu_{\G}$ is the entire cube $[k]^n$.
However, the probability mass is expected to be distributed very unevenly, with large swathes of the cube carrying very little mass.

Fortunately, we do not need to define clusters etc.\ precisely.
Instead, adapting the construction from~\cite{Danny}, we just define a new random variable $\Zkg{\G}$ that comes with a ``hard-wired''
 notion of well-separated clusters.
To be precise, for a graph $G$ denote by $\Sigma_{G,\beta}$ the set of all $\tau\in\cBd$ that enjoy the following property.
\begin{description}
\item[SEP1] 
	for every $i\in[k]$ the set
	$\tau^{-1}(i)$ spans at most $2n\exp(-\beta)k^{-1}\ln k$ edges.
\end{description}
Further, let $\kappa=\ln^{20}k/k$.
We call $\sigma\in\cBd$ {\em separable} if $\sigma\in\Sigma_{G,\beta}$ and if
\begin{description}
\item[SEP2] 
	for every $\tau\in \Sigma_{G,\beta}$ and all $i,j\in[k]$
	such that  $\rho_{ij}(\sigma,\tau)\ge 0.51$ we have $\rho_{ij}(\sigma,\tau)\ge 1-\kappa$.
\end{description}
Let $\cB_{sep}=\cB_{sep}(G,\beta)\subset\cB$ denote the set of all separable maps and define
	\begin{align*}
	\Zkg{\G}&=\sum_{\sigma\in\cB_{sep}(\G,\beta)}\exp(-\beta\cH_{\G}(\sigma)).
	\end{align*}

To elaborate, condition {\bf SEP1} provides that the subgraphs induced on the individual color classes are quite sparse.
Indeed, recalling that each monochromatic edge incurs a ``penalty factor'' of $\exp(-\beta)$, we expect that in a typical sample from the Gibbs
measure the total number of monochromatic edges is about $nd\exp(-\beta)/(2k)$.
Moreover, suppose that $\sigma\in\Sigma_{G,\beta}$ satisfies {\bf SEP2} and $\tau\in\Sigma_{G,\beta}$ is another color assignment.
Let $i,j\in[k]$.
Then {\bf SEP2} provides that there are only two possible scenarios.
\begin{enumerate}
\item[(i)] If $\rho_{ij}(\sigma,\tau)<0.51$, then the color classes $\sigma^{-1}(i)$, $\tau^{-1}(j)$ are ``quite distinct'' and we may think of
	$\sigma,\tau$ as belonging to different ``clusters''.
\item[(ii)] If $\rho_{ij}(\sigma,\tau)\geq0.51$, then in fact $\rho_{ij}(\sigma,\tau)\geq1-\kappa$.
	Thus, the color classes $\sigma^{-1}(i)$, $\tau^{-1}(j)$ are nearly identical.
	Hence, if there is a permutation $\pi:[k]\to[k]$ such that $\rho_{i\pi(i)}(\sigma,\tau)\geq0.51$ for all $i\in[k]$,
	then we may think of $\sigma,\tau$ as belonging to the same ``cluster''.
\end{enumerate}
The upshot is that separability rules out the existence of any ``middle ground'', i.e., we do not have to consider overlaps $\rho$
with entries $\rho_{ij}\in(0,51,1-\kappa)$.

The following proposition, which we prove in \Sec~\ref{sec:sep1mm}, shows that imposing separability has no discernible effect on the first moment.

\begin{proposition}\label{prop:expsepbal}
Assume that $d\in[2(k-1)\ln (k-1),(2k-1)\ln k-2]$ and $\beta\ge \ln k$.
Then
	\begin{align*}
	\Erw[\Zkn{\G}]\sim \Erw[\Zkb{\G}]. 
	\end{align*}
\end{proposition}

The point of working with separable color assignments is that the maximization problem that arises in the second moment computation of $\Zkg{\G}$
comes with further constraints that are not present in~(\ref{eq:2mmopt}).
Specifically, we only need to optimize over $\rho\in\cD$ such that $\rho_{ij}\not\in(0.51,1-\kappa)$ for all $i,j\in[k]$.
In \Sec~\ref{sec:doublystoch} we will use these constraints to derive the following.

\begin{proposition}\label{prop:smm}
Let $d\in[2(k-1)\ln (k-1),d_\star]$ and $\beta\ge \ln k$.
Then
	$
	\frac{1}{n}\ln\Erw[\Zkn{\G}^2]\sim \frac{2}{n}\ln\Erw[\Zkb{\G}].
	$
\end{proposition}

\begin{corollary}
	\label{cor:energy}
	If $d\in[ 2(k-1)\ln(k-1),d_{\star}]$ and $\beta\ge \ln k$, then (\ref{eqThm_main}) holds.
\end{corollary}
	\begin{proof}
On the one hand, Jensen's inequality gives 
	\begin{equation}\label{eqJensen1}
	\Erw[\ln\Zbg{\G}]\leq\ln\Erw[\Zbg{\G}].
	\end{equation}
On the other hand, by  
Propositions~\ref{prop:expsepbal} and~\ref{prop:smm} and the Paley-Zigmund inequality,
		\begin{align}\label{eqJensen2}
		\pr[\Zbg{\G}\geq\Erw[\Zkn{\G}]/2]\geq
		\pr[\Zkn{\G}\geq\Erw[\Zkn{\G}]/2]
			&\geq\frac{\Erw[\Zkn{\G}]^2}{4\Erw[\Zkn{\G}^2]}=\exp(o(n)).
			%\ge \exp(o(n))\Erw[\Zbg{\G}]]\ge\exp(o(n))
		\end{align}
Combining (\ref{eqJensen2}) with \Prop~\ref{prop:firstmoment} , (\ref{eq:balO}) and \Prop~\ref{prop:expsepbal}, we obtain
	\begin{align}\label{eqJensen3}
	\pr[\ln\Zbg{\G}\geq\ln\Erw[\Zbg{\G}]-\ln\ln n]\geq\exp(o(n)).
	\end{align}
Further, (\ref{eqJensen3}) and Fact \ref{lem:logconc} yield $n^{-1}\Erw[\ln\Zbg{\G}]\geq n^{-1}\ln\Erw[\Zbg{\G}]+o(1)$.
Finally, combining this lower bound with the upper bound (\ref{eqJensen1}) completes the proof.
\end{proof}

\noindent
Finally, Theorem \ref{Thm_main} follows from Theorem \ref{Thm_CDGS}, Proposition \ref{prop:maxlnkh0} and Corollary \ref{cor:energy}.

\subsection{Outlook: the condensation phase transition}\label{Sec_cond}
According to non-rigorous physics methods~\cite{pnas,MM} for $d$ only slightly above the bound from \Thm~\ref{Thm_main}
the formula (\ref{eqThm_main}) does not hold for all $\beta>0$ anymore.
While the exact formula is quite complicated (e.g., it involves the solution to a distributional fixed point problem),
the critical degree satisfies $\dc=(2k-1)\ln k-2\ln 2+o_k(1)$.
Thus, for $d>\dc$ there occurs a phase transition at a certain critical inverse temperature $\beta_{k,\mathrm{cond}}(d)$.
The {\em existence} of a critical $\beta_{k,\mathrm{cond}}(d)$ follows from prior results on the random graph coloring problem~\cite{Cond}.
However, the value of $\beta_{k,\mathrm{cond}}(d)$ is not (rigorously) known.

The physics intuition of how this phase transition comes about is as follows.
For $\beta<\beta_{k,\mathrm{cond}}(d)$ the Gibbs measure decomposes into an exponential number of clusters that each have  probability mass $\exp(-\Omega(n))$.
Hence, if we sample $\SIGMA,\TAU$ independently from the Gibbs measure, then most likely they belong to different clusters,
in which case their overlap should be very close to $\bar\rho$.
By contrast, for $\beta>\beta_{k,\mathrm{cond}}(d)$ a bounded number of clusters dominate the Gibbs measure, i.e.,
there are individual clusters whose probability mass is $\Omega(1)$.
In effect, for $\beta>\beta_{k,\mathrm{cond}}(d)$ the overlap of two randomly chosen color assignments is
{\em not} concentrated on the single value $\bar\rho$ anymore, because there is a non-vanishing probability that both belong to the same cluster.
In effect, the second moment method fails.
In fact, we expect that $\Erw[\ln \Zbg{\G}]<\ln\Erw [\Zbg{\G}]-\Omega(n)$ for all $\beta>\beta_{k,\mathrm{cond}}(d)$.

But even the second moment argument for separable color assignments does not quite reach the expected critical degree $\dc$.
Indeed, for $d>(2k-1)\ln k-2+o_k(1)$ the maximum over the set of separable overlaps is attained at
		$\rho_{ij}=\alpha\vecone\{i=j\}+\frac{1-\alpha}{k-1}\vecone\{i\neq j\}$ with $\alpha=1-1/k+o_k(1/k)$.
In terms of the physics intuition, this overlap matrix corresponds to pairs of color assignments that belong to the same cluster.
In other words, the second moment method fails because the {\em expected} cluster size blows up.
A similar problem occurs in the $k$-colorability problem~\cite{Danny}.
There the issue was resolved by explicitly controlling the {\em median} cluster size, which is by an exponential factor smaller
than the expected cluster size~\cite{Cond}.
We expect that a similar remedy applies to the Potts model, although the fact that monochromatic edges are allowed
entails that the proof method from~\cite{Cond} does not apply.
In any case, \Thm~\ref{Thm_main} reduces the task of determining the phase transition to the problem of controlling the median cluster size.

Furthermore, also in the case of degrees above $\dk$ at least the existence of a phase transition has been established rigorously~\cite{CDGS}.
It would be most interesting to see if the present methods can be extended to $d>\dk$ in order to obtain a more precise estimate
of $\beta_{k,\mathrm{cond}}(d)$.

\section{Singly stochastic analysis}\label{sec:singlystoch}

\noindent
We prove \Prop~\ref{prop:maxlnkh0} by way of the following proposition regarding the maximum of $f_{d,\beta}$ over the set of singly-stochastic matrices.

\begin{proposition}
	\label{prop:maxlnkh}
	If $d\in[  2(k-1)\ln (k-1),(2k-1)\ln k-2]$ and $\beta\le \ln k$, then 
	$f_{d,\beta}(\bar{\rho}) > f_{d,\beta}(\rho)$
	for all $\rho\in \cS\setminus\{\bar\rho\}$.
	\end{proposition}

To prove Proposition \ref{prop:maxlnkh} we will closely follow the proof strategy developed for the graph coloring problem in \cite[\Sec~4]{Danny}.
Basically, that argument dealt with optimizing the function $f_{d,\infty}$ (i.e., $c_\beta$ is replaced by $1$) over $\cS$
and we extend that argument to finite values of $\beta$.
In fact, the following monotonicity statement shows that it suffices to prove \Prop~\ref{prop:maxlnkh} for $\beta=\ln k$;
	related monotonicity statements were used in~\cite{h2c} for hypergraph $2$-coloring and in~\cite{Victor} for regular $k$-SAT.

\begin{lemma}
	\label{lem:monotony}
	For all $d>0$, $\beta\geq0$, $\rho\in\cS$ 
	we have
	$$\frac{\partial}{\partial\beta}f_{d,\beta}(\bar\rho)\leq\frac{\partial}{\partial\beta}f_{d,\beta}(\rho)<0.$$
Hence, if $f_{d,\beta'}(\bar{\rho})\ge f_{d,\beta'}(\rho)$ for $\beta'\in[0,\infty]$, then $f_{d,\beta}(\bar{\rho})\ge f_{d,\beta}(\rho)$ for all $\beta<\beta'$.
	\begin{proof}
		Differentiating by $\beta$ reveals that $\beta\mapsto f_{d,\beta}(\rho)$ is monotonous.
		\begin{align}
		\frac{\partial}{\partial \beta}f_{d,\beta}(\rho)=-\frac{d}{2}\frac{\frac{2}{k}-\|\rho\|_2^2\frac{2c_\beta}{k^2} e^{-\beta}}{1-\frac{2}{k}c_\beta+\frac{\|\rho\|_2^2}{k^2}c_\beta^2}<0.\label{eq:monbeta}
		\end{align}
		Setting $y=\|\rho\|_2^2$ and construing $\frac{\partial}{\partial \beta}f_{d,\beta}(\rho)$ as a map of $y$,
		\begin{align*}
		\phi:[1,k]\to \RR,\quad \phi(y)\mapsto -\frac{d}{2}\frac{\frac{2}{k}-y\frac{2c_\beta}{k^2} e^{-\beta}}{1-\frac{2}{k}c_\beta+\frac{y}{k^2}c_\beta^2},
		\end{align*}
		differentiating $\frac{\partial}{\partial \beta}f_{d,\beta}(\rho)$ by $y$, we obtain
		\begin{align}
		\frac{\partial}{\partial y}\phi(y)&=\frac{\frac{1}{k^2}2c_\beta e^{-\beta}\left(1-\frac{2}{k}c_\beta+\frac{y}{k^2}c_\beta^2\right)-\left(-\frac{2}{k}e^{-\beta}+\frac{y}{k^2}2c_\beta e^{-\beta}\right)\frac{c_\beta^2}{k^2}}{\left(1-\frac{2}{k}c_\beta+\frac{y}{k^2}c_\beta^2\right)^2}\nonumber \\
		&=\frac{\frac{2c_\beta e^{-\beta}}{k^2}\left(1-\frac{c_\beta}{k}\right)+y\frac{2c_\beta^3e^{-\beta}}{k^3}\left(1-\frac{1}{k}\right)}{\left(1-\frac{2}{k}c_\beta+\frac{y}{k^2}c_\beta^2\right)^2}\ge 0\qquad \text{ for }y\in[1,k].\label{eq:mony}
		\end{align}
		Hence, $y\mapsto\frac{\partial}{\partial \beta}f_{d,\beta}(\rho)$ has a global minimum at $y=1$.  Because $\|\rho\|_2^2=1$ is only the case for $\rho=\bar{\rho}$ the combination of \eqref{eq:monbeta} and \eqref{eq:mony} yields the assertion.
	\end{proof}
\end{lemma}

The following basic observation concerning the partial derivatives of $f_{d,\beta}$ is reminiscent of~\cite[\Lem~4.11]{Danny}.

\begin{claim}\label{claim:bestrho}
Let $\rho\in\cS$. With $i,j,l\in[k]$ such that $\rho_{il},\rho_{ij}>0$ set $\delta=\rho_{il}-\rho_{ij}$.
	\begin{enumerate}[i)]
		\item Then 
		\[\sign\left(\frac{\partial}{\partial \rho_{ij}}f_{d,\beta}(\rho)-\frac{\partial}{\partial \rho_{il}}f_{d,\beta}(\rho)\right)=\sign\left(1+\frac{\delta}{\rho_{ij}}-\exp\left(\frac{dc_\beta\delta}{k-2c_\beta+c_\beta^2\|\rho\|_2^2/k}\right)\right).\]
		\item If ${\partial E(\rho)}/{\partial\rho_{ij}}<1/k$ then there is $\delta^\ast>0$ such that for all $0<\delta<\delta^\ast$ 
		\begin{align}
		1+\frac{\delta}{\rho_{ij}}-\exp\left(\frac{dc_\beta\delta}{k-2c_\beta+c_\beta^2\|\rho\|_2^2/k}\right)>0.\label{eq:sign}
		\end{align}
		If ${\partial E(\rho)}/{\partial\rho_{ij}}\ge1/k$, the left hand side of \eqref{eq:sign} is negative for all $\delta>0$.
	\end{enumerate}
\end{claim}
\begin{proof}
By \eqref{eq:derivH}, \eqref{eq:derivE} and the choice of $\delta$,
	\begin{align}
	\frac{\partial}{\partial\rho_{ij}}f_{d,\beta}(\rho)-\frac{\partial}{\partial\rho_{il}}f_{d,\beta}(\rho)=\frac{1}{k}\left[\ln\left(1+\frac{\delta}{\rho_{ij}}\right)-\frac{dc_\beta^2\delta}{k-2c_\beta+c_\beta\|\rho\|_2^2/k}\right].\label{eq:derivF}
	\end{align} 
	The first part of the claim follows because the signs of the terms in \eqref{eq:derivF} are invariant under exponentiation of the minuend $\phi(\delta)=\ln(1+\delta/\rho_{ij})$ and subtrahend $\psi(\delta)=dc_\beta^2\delta/(k-2c_\beta+c_\beta\|\rho\|_2^2/k)$. The second part follows from the observation that the linear function $\exp(\phi):\RR^+\to\RR$ intersects at most once with the strictly convex function $\exp(\psi):\RR^+\to\RR$. This is only the case if the derivative of $\exp(\phi)$ in $\delta=0$ is strictly greater than that of $\exp(\psi)$.
\end{proof}

\noindent
The following lemma provides a general ``maximum entropy'' principle that we will use repeatedly (cf.~\cite[\Prop~4.7]{Danny}).

\begin{lemma}
	\label{lem:smax}
	Let $d\le (2k-1)\ln k$ and $\beta>0$. For $\rho\in\cS$, a fixed row $i$ and a set of columns $J\subset[k]$, 
	set
	$\hat\rho_{ab}=\sum_{j\in J}\rho_{ij}/|J|$ for all $(a,b)\in\{i\}\times J$ and $\hat{\rho}_{ab}=\rho_{ab}$ for all $(a,b)\notin\{i\}\times J$. Let $\lambda\ge3\ln\ln k/\ln k$. If $|J|\ge k^\lambda$ and $\max_{j\in J}\rho_{ij}<\lambda/2-\ln\ln k/\ln k$, then $f_{d,\beta}(\hat\rho)> f_{d,\beta}(\rho)$ if $\rho\neq\hat{\rho}$.
\end{lemma}
\begin{proof}
	We may assume that $0\le\min_{j\in J}\rho_{ij}<\max_{j\in J}\rho_{i,j}$. Otherwise, we would have $\hat{\rho}=\rho$ and there is nothing to prove.
	Now let 
	\begin{align*}
	\cS_\rho=\left\{\tilde{\rho}\,:\,\tilde{\rho}_{ab}=\rho_{ab}\text{ for all }(a,b)\in\{i\}\times J\text{ and }\max_{j\in J}\tilde\rho_{ij}\le \max_{j\in J}\rho_{i,j} \right\}
	\end{align*}
	denote the set of all possible overlaps. $\cS_\rho$ is a closed subset of $\cS$ and therefore contains a maximal overlap $\check{\rho}\in\arg\max_{\tilde{\rho}\in\cS}f_{d,\beta}(\tilde{\rho})$. 
	Evidently the derivative of $H$ tends to infinity as $\rho_{ij}$ tends to zero, while the derivative of $E$ remains bounded. Therefore in a maximal overlap each entry $\check{\rho}_{ij}$, $j\in J$ is positive. As a whole, we know that $0<\min_{j\in J}\check{\rho}_{ij}\le \max_{j\in J}\check{\rho}_{ij}\le 1$. By means of Claim \ref{claim:bestrho} it remains to show that $\check\delta=\max_{j\in J}\check{\rho}_{ij}-\min_{j\in J}\check{\rho}_{ij}=0$.

	Let $a\in J$ denote the index of $\check\rho_{ia}=\min_{j\in J}\check\rho_{ij}$. Because $|J|\check\rho_{ia}\le \sum_{j\in J}\check\rho_{ij}$ and $d\le 2k\ln k-\ln k$, we have
	 \[\frac{1}{\check\rho_{ia}}\ge |J|\ge k^\lambda\ge 3\ln k>2\ln k\left(\frac{kc_\beta^2}{k-2c_\beta+c_\beta^2/k}\right)\ge \frac{k}{\check\rho_{ia}}\frac{\partial}{\partial\check\rho_{ia}}E(\check{\rho}).\]
	As $\hat{\delta}=\lambda/2-\ln\ln k/\ln k$, $\|\check\rho\|_2^2\ge1$ and $d\le 2k\ln k-\ln k$,
	\begin{align*}
	&\exp\left(\frac{dc_\beta^2\hat\delta}{k-2c_\beta+c_\beta^2\|\check\rho\|_2^2/k}\right)\le\exp\left(\frac{d\hat\delta}{k(1-c_\beta/k)^2}c_\beta^2\right) 
	\le \exp(2\hat{\delta}\ln k)\\
	\le& k^{\lambda}\ln^{-2}k\le |J|\ln^{-2}k\le \frac{1}{\check{\rho}_{ia}}\frac{1}{\ln^2k}<\frac{1}{\check{\rho}_{ia}}\frac{\ln\ln k}{2\ln k}\le\frac{1}{\check{\rho}_{ia}}\left(\frac{\lambda}{2}-\frac{\ln\ln k}{\ln k}\right)\le \frac{\hat{\delta}}{\check{\rho}_{ia}}
	\end{align*}
	confirms that  \[\sign\left(1+\frac{\delta}{\check{\rho}_{ia}}-\exp\left(\frac{dc_\beta\delta}{k-2c_\beta+c_\beta^2\|\check\rho\|_2^2/k}\right)\right)=1\]
	holds for any $\delta<\hat{\delta}$. Suppose that $\check{\delta}>0$. Then $0<\delta\le\max_{j\in J}\check{\rho}_{ij}\le\hat{\delta}$ and Claim \ref{claim:bestrho} imply that a matrix $\check\rho'$ obtained from $\check\rho'$ by decreasing $\max_{j\in J}\check{\rho}_{ij}$ by a sufficiently small $\xi>0$ and increasing $\check{\rho}_{ia}$ by the same value $\xi$ results in $f_{d,\beta}(\check{\rho}')>f_{d,\beta}(\check{\rho})$, which contradicts the maximality of $\check\rho$. Hence, a maximal overlap $\rho$ satisfies $\check\delta=\max_{j\in J}\check{\rho}_{ij}-\min_{j\in J}\check{\rho}_{ij}=0$ for any $i$, $J$ chosen according to our assumption.
\end{proof}

In order to achieve a global bound on $\max_{\rho\in \cSd}f_{d,\beta}(\rho)$ we need to pin down the structure of a maximizing matrix $\rho$.
To this end, the following elementary fact is going to be useful.

\begin{fact}[{\cite[Lemma 4.15]{Danny}}]
		\label{lem:xi}
		Let $\xi:\eps\in(0,k/2)\mapsto k^{2\eps/k}(\eps^{-1}-k^{-1})$. Let $\mu=\frac{k}{2}(1-\sqrt{1-2/\ln k})$. Then $\xi$ is decreasing on the interval $(0,\mu)$ and increasing on $(\mu,k/2)$. Furthermore, we have 
		$
		-1/2\le \xi'(\eps)\le -3/2\text{ for }b\in(0.99,1.01).
		$
	\end{fact}
	
\noindent
The following lemma rules out the possibility that the maximizer of $f_{d,\beta}$ has an entry close to $1/2$
	(cf.\ \cite[\Lem~4.13]{Danny}).

\begin{lemma}
	\label{lem:NoHalfEntries}
	Let $\beta>0$ and $d=2k\ln k -c$, where $c=O_k(\ln k)$. If $\rho\in\cSd$ has an entry $\rho_{ij}\in[0.49,0.51]$, then there is $\rho'\in \cSd$ such that $f_{d,\beta}(\rho')\ge f_{d,\beta}(\rho)+\frac{\ln k}{5k}$.
	\begin{proof}
		By means of Lemma \ref{lem:smax} we will specify $\rho'$ and provide above bound for $f_{d,\beta}(\rho)-f_{d,\beta}(\rho')$ in a distinction of two cases. Without loss of generality we may assume that the entry in the interval $[0.49,0.51]$ is $\rho_{11}$. Suppose $\rho$ maximizes $f_{d,\beta}$ subject to the condition that $\rho_{11}\in[0.49,0,51]$. 
		
		For the first case, suppose that $\rho_{1j}<0.49$ for all $j\ge 2$. By setting $J=\{2,\ldots,k\}$ and $\lambda=\ln (k-1)/\ln k$ in Lemma \ref{lem:smax}, we have $\rho_{1j}=(1-\rho_{11})/(k-1)$ for all $j\ge 2$. Let $\rho'$ denote the matrix obtained from $\rho$ by setting $\rho'_1=(1/k,\ldots,1/k)$ and $\rho_i'=\rho_i$ for $i\ge 2$. In the following assume that $k$ is sufficiently large. By Fact \ref{fact:entropy} we have
		\begin{align*}
			H(\rho_1)\le h(\rho_{11})+(1-\rho_{11})\ln(k-1)\le \ln 2 + 0.51\ln k.
		\end{align*}
		Consequently
		\begin{align}
			H(k^{-1}\rho'_1)-H(k^{-1}\rho_1)\ge \frac{0.48 \ln k}{k}.\label{eq:Case3Hdiff}
		\end{align}
		In comparison, the Frobenius norm of $\rho_1$ is bounded by
		\begin{align*}
			\|\rho_1\|_2^2\le 0.51^2+(k-1)\left(\frac{0.51}{k-1}\right)^2\le 0.261,
		\end{align*}
		while 
		\begin{align}
			\label{eq:derivEy}
			\frac{\partial}{\partial \|\rho\|_2^2}E(\rho)&=\frac{d}{2k^2}\frac{c_\beta^2}{1-2/kc_\beta+\|\rho\|_2^2/k^2c_\beta^2} =\frac{2k\ln k + O_k(\ln k)}{2k}O_k\left(\frac1k\right) 
			\le\frac{\ln k}{k}\left(1+O_k\left(\frac{1}{k}\right)\right).
		\end{align}
		Therefore 
		\begin{align}
			E(\rho)-E(\rho')\le \frac{0.262\ln k}{k}.\label{eq:Case3Ediff}
		\end{align}
		The combination of \eqref{eq:Case3Hdiff} and \eqref{eq:Case3Ediff} verifies
		\begin{align*}
			f_{d,\beta}(\rho')\ge f_{d,\beta}(\rho)+0.218\frac{\ln k}{k}\ge  f_{d,\beta}(\rho)+\frac{\ln k}{5k}
		\end{align*}
		for $\beta\ge\ln k$. By Lemma \ref{lem:monotony} 
		\begin{align}
			f_{d,\beta}(\rho')\ge f_{d,\beta}(\rho) +
			\frac{\ln k}{5k} \label{eq:case3ineq}
		\end{align} holds for any $0\le \beta \le \ln k$.
		Finally we show \eqref{eq:case3ineq} for the case that a row consists of two entries greater than 0.49. Without loss of generality we may assume that $\rho_{11}\ge\rho_{12}\ge 0.49$ and $\rho_{1j}<0.02$ for $j\ge 3$. Lemma \ref{lem:smax} with parameters $J=\{2,\ldots,k\}$ and $\lambda=\ln (k-1)/\ln k$ gives  $\rho_{1j}=(1-\rho_{11}-\rho_{12})/(k-2)$ for all $j\ge 3$. Hence, for sufficiently large $k$
		\begin{align*}
			H(\rho_1)\le h(\rho_{11}) + h(\rho_{12})+(0.02)\ln(k-2)\le 2\ln 2 + 0.02\ln k\le 0.03\ln k.
		\end{align*}
		Moreover the norm is bounded by
		\begin{align*}
			\|\rho_1\|_2^2= \rho_{11}^2+\rho_{12}^2+(k-2)\left(\frac{1-\rho_{11}-\rho_{12}}{k-2}\right)^2\le 0.501.
		\end{align*}
		Consequently
		\begin{align}
			E(\rho)-E(\rho')&\le \frac{0.51 \ln k}{k},\label{eq:Case31Ediff}\\
			H(k^{-1}\rho'_1)-H(k^{-1}\rho_1)&\ge 0.97\frac{\ln k}{k}. \label{eq:Case31Hdiff}
		\end{align}
		The combination of  \eqref{eq:Case31Ediff} and \eqref{eq:Case31Hdiff} yields \eqref{eq:case3ineq} for $\beta\ge\ln k$. By Lemma \ref{lem:monotony} the assertion follows for $0\le \beta\le \ln k$.
	\end{proof}
\end{lemma}

\noindent
Generalizing \cite[\Lem~4.16]{Danny}, as  a next step we characterize the structure of the local maxima of $f_{d,\beta}$ on $\cS$.

\begin{lemma}
	\label{lem:FormMaxMatrix}
	Let $\beta>0$ and $d= 2k\ln k -c$, where $c=O_k(\ln k)$. Let $\rho\in \cSd$.
	\begin{enumerate}[(1)]
		\item Suppose that row $i\in[k]$ has no entries in $[0.49,0.51]$ and $\rho_{ij}\le 0.49$ for all $j\in[k]$. Let $\rho'$ be the stochastic matrix with entries
		\begin{align} 
			\rho'_{hj}=\rho_{hj}\text{ and }\rho'_{ij}=\frac{1}{k}\quad\text{ for all }j\in [k], h\in [k]\setminus\{i\}.\label{eq:bettermatrix}
		\end{align}
		Then $f_{d,\beta}(\rho)\le f_{d,\beta}(\rho')$.
		\item Suppose that row $i\in[k]$ has no entries in $[0.49,0.51]$ and $\rho_{ij}\ge 0.51$ for some $j\in [k]$. Then there is a number $\alpha=\frac1k+\tilde{O}_k(1/k^2)$ such that for the stochastic matrix $\rho''$ with entries
		\begin{align}
			\rho''_{hj}=\rho_{hj} \text{ and }\rho''_{ii}=1-\alpha, \rho''_{ih}=\frac{\alpha}{k-1}\quad\text{ for all }j\in [k], h\in [k]\setminus\{i\}\label{eq:bettermatrix2}
		\end{align}
		we have $f_{d,\beta}(\rho)\le f_{d,\beta}(\rho'')$.
		\item Let $\beta\le \ln k$. Suppose that row $i\in[k]$ has an entry $\rho_{ij}\in[0.49,0.51]$. Then the matrix $\rho'$ with \eqref{eq:bettermatrix} satisfies $f_{d,\beta}(\rho)\le f_{d,\beta}(\rho')$.
	\end{enumerate}
	\begin{proof}
		Claim (1) is an immediate consequence of Lemma \ref{lem:smax} when setting $J=[k]$, $\lambda=1$ and applying the $\rho\mapsto\hat\rho$ operation on the $i$-th row.
		
		For Claim (2) we may again assume that $i=j=1$ and therefore $\rho_{11}\ge0.51$. Let $\hat{\rho}\in\cS$ maximize $f_{d,\beta}$ subject to the conditions that $\hat{\rho}$ coincides with $\rho$ everywhere but in the first row and $\hat{\rho}_{11}\ge0.51$. A necessary condition for $\hat{\rho}$ to be maximal is that the mass in the remaining open entries is equally distributed. $\hat{\rho}_{11}\ge0.51$ implies that for all $j\ge2$ the entries $\hat{\rho}_{1j}$ are bounded by 0.49. Setting $\lambda=\ln(k-1)/\ln k$, Lemma \ref{lem:smax} applies to row $i=1$ and $J=\{2,\ldots,k\}$ confirming that for all $j\ge2$ we have $\hat{\rho}_{1j}=(1-\hat{\rho}_{11})/(k-1)$.
		
		Let $0\le \eps\le 0.49k$ be such that $\hat{\rho}_{11}=1-\eps/k$. To prove the assertion we need to show that $\eps=1+\tilde{O}_k(1/k)$. Set $\delta=\hat{\rho}_{11}-\hat{\rho}_{12}$. Then because $\hat{\rho}$ maximizes $f_{d,\beta}$ Claim \ref{claim:bestrho} implies that
		\begin{align}
			\text{either }\eps\in\{0,0.49k\}\text{, or } 1 +\frac{\delta}{\hat{\rho}_{12}}=\exp\left(\frac{dc_\beta^2 \delta}{k-2c_\beta+c_\beta^2\frac{\|\hat{\rho}\|_2^2}{k}}\right).\label{eq:epscases}
		\end{align}
		Equations \eqref{eq:derivF} and \eqref{eq:derivH} show that $\partial/\partial \rho_{11}H(\rho_1)$ tends to $-\infty$ as $\rho_{11}$ tends to 1, while $\partial/\partial \rho_{11}E(\rho_1)$ remains bounded. Hence, a maximal $\hat\rho$ is bound to satisfy $\eps>0$.
		
		By $\|\hat{\rho}\|_2^2\ge1$ we have $k-2c_\beta+c_\beta\frac{\|\hat{\rho}\|_2^2}{k}\ge k(1-c_\beta/k)^2$. Moreover we have $\delta=\hat{\rho}_{11}-O_k(1/k)$ due to all entries in the first row being $(1-\hat{\rho}_{11})/(k-1)$. With $d=2k\ln k + O_k(\ln k)$ and $\beta\ge \ln k$ we obtain
		\begin{align*}
			\exp\left(\frac{dc_\beta^2 \delta}{k-2c_\beta+c_\beta^2\frac{\|\hat{\rho}\|_2^2}{k}}\right)=k^{2\hat{\rho}_{11}}\left(1+\tilde{O}_k(1/k)\right) =k^{2(1-\eps/k)}\left(1+O_k(1/k)\right)
		\end{align*}
		and
		\begin{align*}
			1 +\frac{\delta}{\hat{\rho}_{12}}=\frac{\hat{\rho}_{11}}{\hat{\rho}_{12}}=\frac{(k-1)\hat{\rho}_{11}}{1-\hat{\rho}_{11}}=k^2(1/\eps-1/k)(1+O_k(1/k)).
		\end{align*}
		Thus, setting $\xi:\eps\mapsto k^{2\eps/k}(1/\eps-1/k)$ there is $\eta=O_k(\ln k/k)$ such that
		\begin{align}
			(1-\eta)\xi(\eps)\le \left(1+\frac{\delta}{\hat{\rho}_{12}}\right)\exp\left(\frac{dc_\beta^2 \delta}{k-2c_\beta+c_\beta\frac{\|\hat{\rho}\|_2^2}{k}}\right)\le 	(1+\eta)\xi(\eps).\label{eq:xibounds}
		\end{align}
		Fact \ref{lem:xi} reveals that $\xi$ has a unique local minimum in $\mu=\frac{k}{2}(1-\sqrt{1-2/\ln k})$ while $\xi$ is decreasing on $(0,\mu)$ and increasing on $(\mu,k/2)$. Furthermore we have $\xi(\eps)\in[-3/2,-1/2]$ for $\eps\in(0.99,1.01)$. Therefore, setting $\gamma=\ln^2k/k$, we have
		\begin{align*}
			\xi(\eps)\le\begin{cases}
				\xi(0.49k)\le k^{0.98}\left(\frac{1}{0.49k}-\frac{1}{k}\right)<\frac{1}{1+\eta}&\text{ for }\eps\in[\mu,0.49k]\\
				\xi(1+\gamma)<\frac{1}{1+\eta}&\text{ for }\eps\in[1+\gamma,\mu]
			\end{cases}
		\end{align*}
		and
		\begin{align*}
			\xi(\eps)\ge \xi(1-\gamma)>\frac{1}{1-\eta},\qquad\text{ for } \eps\in(0,1-\gamma).
		\end{align*}
		These bounds applied to \eqref{eq:xibounds} yield
		\begin{align}
			1 +\frac{\delta}{\hat{\rho}_{12}}-\exp\left(\frac{dc_\beta^2 \delta}{k-2c_\beta+c_\beta\frac{\|\hat{\rho}\|_2^2}{k}}\right)\begin{cases}>0&\text{ for }\eps\in(0,1-\gamma),\\
				<0&\text{ for }\eps\in[1+\gamma,0.49k].
			\end{cases}\label{eq:epsineq}
		\end{align}
		Altogether \eqref{eq:epscases} and \eqref{eq:epsineq} with $\eps>0$ imply $\eps=1+\tilde{O}_k(1/k)$ and therefore  $\hat{\rho}_{11}=1-1/k+\tilde{O}_k(1/k^2)$ by Claim \ref{claim:bestrho}. Hence $\hat\rho$ satisfies \eqref{eq:bettermatrix2} and $f_{d,\beta}(\hat{\rho})\ge f_{d,\beta}(\rho)$ for any $\beta\ge \ln k$. By Lemma \ref{lem:monotony} $f_{d,\beta}(\hat{\rho})\ge f_{d,\beta}(\rho)$ holds for any $0\le \beta\le \ln k$ as well.
		
		By definition of $\rho'$ Claim (3) is a Corollary of Lemma \ref{lem:NoHalfEntries}.
	\end{proof}
\end{lemma}

\noindent
The following Lemma, which extends~\cite[\Lem~4.14]{Danny} to finite $\beta$,
 estimates the function values attained at points near the ``candidate maxima'' from \Lem~\ref{lem:FormMaxMatrix}.
	
\begin{lemma}
	\label{lem:barmax}
	Let $\rho_s$ denote the matrix whose the top $s$ rows coincide with the identity matrix and whose last $k-s$ rows coincide with $\bar{\rho}$.
	If $\beta=\ln k$ and $d\le (2k-1)\ln k$ then $f_{d,\beta}(\bar{\rho})>f_{d,\beta}(\rho_s)$ for all $s=1,\ldots,k$.
\end{lemma}
\begin{proof}
We have
	\begin{align*}
	H(k^{-1}\bar\rho)&=\ln k + \frac{1}{k}\sum_{i=1}^kH(\rho_i)=2\ln k,%\label{eq:barmaxHB}
	&
	E(\bar\rho)&=\frac{d}{2}\ln\left[1-\frac{2}{k}c_\beta+\frac{1}{k^2}c_\beta^2\right]=d\ln\left[1-\frac{c_\beta}{k}\right].%\label{eq:barmaxEB}
	\end{align*}
Further,
	\begin{align}
	H(k^{-1}\rho_s)&=\ln k + \frac{1}{k}\sum_{i=1}^kH(\rho_i)=\ln k + \frac{k-s}{k}\ln k,\label{eq:barmaxHS}\\
	E(\rho_s)&=\frac{d}{2}\ln\left[1-\frac{2}{k}c_\beta+\left(\frac{k-s}{k}+s\right)\frac{c_\beta^2}{k^2}\right].\label{eq:barmaxES}
	\end{align}
Hence,
	\begin{align*}
	f_{d,\beta}(\bar{\rho})&=2\ln k + d\ln[1-c_\beta/k],&
	f_{d,\beta}(\rho_s)&=\frac{2k-s}{k}\ln k + \frac{d}{2}\ln\left[1-\frac{2}{k}c_\beta+\left(\frac{k-s}{k}+s\right)\frac{c_\beta^2}{k^2}\right].
	\end{align*}
	The assertion $f_{d,\beta}(\bar{\rho})>f_{d,\beta}(\rho_s)$ holds iff $H(k^{-1}\bar{\rho})-H(k^{-1}\rho_s)=\frac{s}{k}\ln k > E(\rho_s)-E(\bar{\rho})$, i.e.
	\begin{align}
	E(\rho_s)-E(\bar{\rho})=\frac{d}{2}\ln\left[\frac{1-\frac{2}{k}c_\beta+\left(\frac{k-s}{k}+s\right)\frac{c_\beta^2}{k^2}}{(1-\frac{c_\beta}{k})^2}\right]=\frac{d}{2}\ln\left[1+\frac{\left(s-\frac{s}{k}\right)\frac{c_\beta^2}{k^2}}{\left(1-\frac{c_\beta}{k}\right)^2}\right]<\frac{s}{k}\ln k.\label{eq:barmaxineq}
	\end{align}
	Setting $x=\left(s-{s}/{k}\right){c_\beta^2}/{k^2}\left(1-{c_\beta}/{k}\right)^{-2}$ a mercator series expansion 
	\begin{align*}
	\frac{d}{2}\ln(1+x)=\frac{d}{2}\left[x-\frac{x^2}{2}+O_k(x^3)\right]\le \frac{2k\ln k-\ln k}{2}\left[x-\frac{x^2}{2}\right]=\ln k\left[kx-k\frac{x^2}{2}-\frac{x}{2}+\frac{x^2}{4}\right]
	\end{align*}
	along with the representation
	\begin{align*}
	\frac{\left(s-\frac{s}{k}\right)\frac{c_\beta^2}{k^2}}{\left(1-\frac{c_\beta}{k}\right)^2}&=\frac{c_\beta^2}{k}\frac{s}{k}\left(1-\frac{1}{k}\right)\frac{1}{(1-c_\beta/k)^2}=\frac{c_\beta^2}{k}\frac{s}{k}\left(1-\frac{1}{k}\right)\left(1+2c_\beta/k+O_k(1/k^2)\right)\\
	&=\frac{1}{k}\frac{s}{k}\left(1-\frac{2}{k}+O_k(k^{-2})\right)\left(1+\frac{1}{k}+O_k(k^{-2})\right)\tag*{[as $\beta=\ln k$]}
	\end{align*}
	reduces the proof to validating the inequality
	\begin{align}
	\frac{(k-1/2)}{k}\left(1-\frac{2}{k}+O_k(k^{-2})\right)\left(1+\frac{1}{k}+O_k(k^{-2})\right)+\frac{(1/4-k/2)}{k^2}\left[\left(1-\frac{2}{k}+O_k(k^{-2})\right)\left(1+\frac{1}{k}+O_k(k^{-2})\right)\right]^2\frac{s}{k}<1
	.\label{eq:barmaxcrit}
	\end{align}
	This is indeed true, since the first summand is bounded by $1-k^{-2}$ and the second summand is negative.
\end{proof}

\begin{corollary}
	\label{cor:monotonb}
	With $\rho_s$ defined as in Lemma \ref{lem:barmax} the inequality $f_{d,\beta}(\bar{\rho})>f_{d,\beta}(\rho_s)$ holds for all $s<k$ and $0<\beta\le \ln k$.
\end{corollary}

\begin{proof}[Proof of Proposition \ref{prop:maxlnkh}]
	In the case $\beta=0$ we have $f_{d,\beta}(\rho)=H(k^{-1}\rho)$. On $[0,1]^{k\times k}\supset \cS$ the entropy function is maximized by the uniform distribution on $[k]^2$, i.e. the matrix $\bar{\rho}$.
	Consider the case $0<\beta\le\ln k$. Because $\rho$ is stochastic each row of $\rho$ has at most one entry greater than 0.51. We call $\rho$ $s$-stable if there are precisely $s$ rows with entries greater than 0.51. Let $\rho_s$ denote the matrix where the top $s$ rows coincide with the identity matrix and the last $k-s$ rows with $\bar{\rho}$. For any $s\in\{0,1,\ldots,k\}$ and any $s$-stable matrix $\rho$, using Lemma \ref{lem:FormMaxMatrix} we obtain a matrix $\rho'$ such that $f_{d,\beta}(\rho')\ge f_{d,\beta}(\rho)$ where $\rho'$ is achieved by moving from $\rho$ in direction $\rho_s$. Together with Corollary \ref{cor:monotonb} this yields the assertion.
\end{proof}	

\begin{proof}[Proof of Proposition \ref{prop:maxlnkh0}]
	For any choice of $n, \beta$ or $d$ Jensen's inequality shows
	\begin{align}
	\frac{1}{n}\ln\Erw[\Zbg{\G}]\ge\frac{1}{n}\Erw[\ln\Zbg{\G}].\label{eq:jensen}
	\end{align}
	We claim that $d\in[2(k-1)\ln (k-1), (2k-1)\ln k- 2]$ and $\beta\le\ln k$ allows for
	\begin{align}
	\frac{1}{n}\ln\Erw[\Zbg{\G}]\le\frac{1}{n}\Erw[\ln\Zbg{\G}]+o(1).\label{eq:reversal}
	\end{align}
	By \eqref{eq:balO}, there is $C_b>$ such that 
	\begin{align}
	\Erw[\Zbg{\G}]\le C_b\Erw[\Zkb{\G}].\label{eq:OCB}
	\end{align} Hence, combining Propositions \ref{prop:exppairs} and \ref{prop:maxlnkh} we have 
	\begin{align*}\Erw[\Zkb{\G}^2]=\sum_{\rho\in\cR} \exp(nf_{d,\beta}(\rho)+o(n))\le \exp(o(n))\exp(nf_{d,\beta}(\bar{\rho})/2)\le \exp(o(n))\Erw[\Zkb{\G}]^2.
	\end{align*}
	Analogously to the proof of Corollary \ref{cor:energy} we apply the Paley-Zigmund inequality and obtain
	\begin{align*}
	\liminf_{n\to\infty}\pr[n^{-1}\ln(\Zbg{\G})\ge n^{-1}\ln\Erw[\Zbg{\G}]-o(1)]\ge \exp(o(n)).
	\end{align*} The concentration result in Fact \ref{lem:logconc} therefore yields
	$\frac{1}{n}\Erw [\ln \Zbg{\G}]\ge\frac{1}{n}\ln\Erw [\Zbg{\G}] -o(1)$.
\end{proof}
	
\section{High degree, low temperature: the first moment} \label{sec:sep1mm}

\noindent
{\em Throughout this section we assume that $d\in[2(k-1)\ln (k-1),(2k-1)\ln k-2]$ and $\beta\ge \ln k$.}
In this section we prove \Prop~\ref{prop:expsepbal}.
The principal tool is going to be the following experiment called the
\emph{planted model};
	similar constructions for hypergraph $2$-coloring or $k$-SAT played an important role in~\cite{Victor,h2c}.

\begin{description}
	\item[PM1] Choose a map $\hat\SIGMA:[n]\to[k]$ uniformly at random.
	\item[PM2] Letting
	\begin{align*}
	p_1=\frac{dk\exp(-\beta)}{n(k-c_\beta)},\qquad p_2=\frac{dk}{n(k-c_\beta)},
	\end{align*}
	obtain a random graph $\hat\G$ on $[n]$ by independently including every edge $\{v,w\}$ of the complete graph such that
		$\hat\SIGMA(v)\neq\hat\SIGMA(w)$ with probability $p_2$ and every edge $\{v,w\}$ such that
		$\hat\SIGMA(v)=\hat\SIGMA(w)$ with probability $p_1$.
\end{description}

\noindent
The following lemma sets out the connection between the planted model and the first moment.

\begin{lemma}\label{lem:contig}
If $\cA$ is a set of graph/color assignment pairs $(G,\sigma)$ such that
	$\pr\brk{(\hat\G,\hat\SIGMA)\in\cA|\hat\SIGMA\in\cB}=o(n^{-1/2})$, then
	$$\Erw\sum_{\sigma\in\cB}\exp(-\beta\cH_{\G}(\sigma))\vecone_{\{(\G,\sigma)\in\cA\}}=o(\Erw[\Zkb{\G}]).$$
\end{lemma}
\begin{proof}
Because $k^{-1}p_1+(1-1/k)p_2=d/n$, the expected number of edges of $\hat\G$ is $m+O(\sqrt n)$.
Hence, the assumption $\pr\brk{(\hat\G,\hat\SIGMA)\in\cA|\hat\SIGMA\in\cB}=o(n^{-1/2})$ implies that
	\begin{align}\label{eqcontig1}
	\pr\brk{(\hat\G,\hat\SIGMA)\in\cA|\hat\SIGMA\in\cB,\,|E(\hat\G)|=m}=o(1).
	\end{align}
Writing out the l.h.s.\ of (\ref{eqcontig1}), we obtain
	\begin{align*}
	\pr\brk{(\hat\G,\hat\SIGMA)\in\cA|\hat\SIGMA\in\cB,\,|E(\hat\G)|=m}\\
		&\hspace{-4cm}=
		\Theta(k^{-n})\sum_{(G,\sigma)\in\cA,\sigma\in\cB,|E(G)|=m}
			\frac{p_1^{\cH_G(\sigma)}p_2^{m-\cH_G(\sigma)}(1-p_1)^{\Forb(\sigma)-\cH_G(\sigma)}(1-p_2)^{\bink n2-\Forb(\sigma)-m+\cH_G(\sigma)}}
				{\pr\brk{|E(\hat\G)=m|}}\\
		&\hspace{-4cm}=\frac{\Theta(k^{-n})}{(1-c_\beta/k)^m}\bcfr dn^m
			\sum_{(G,\sigma)\in\cA,\sigma\in\cB,|E(G)|=m}
				\frac{\exp(-\beta \cH_G(\sigma))
					(1-p_1)^{k^{-1}\bink n2-\cH_G(\sigma)}(1-p_2)^{(1-k^{-1})\bink n2-m+\cH_G(\sigma)}
					}{\pr\brk{\Bin(\bink n2,d/n)=m}};
	\end{align*}
in the last step we used (\ref{eqForb0}) and the observation that  $k^{-1}p_1+(1-1/k)p_2=d/n$.
Further, combining the above with (\ref{eq:balO}), we get
	\begin{align*}
	\pr\brk{(\hat\G,\hat\SIGMA)\in\cA|\hat\SIGMA\in\cB,\,|E(\hat\G)|=m}\Erw[\Zkb{\G}]\\
		&\hspace{-4cm}=
		\Theta(1)\bink{\bink n2}{m}^{-1}
		\sum_{(G,\sigma)\in\cA,\sigma\in\cB,|E(G)|=m}\exp(-\beta\cH_G(\sigma))\bcfr{1-p_1}{1-p}^{k^{-1}\bink n2-\cH_G(\sigma)}
			\bcfr{1-p_2}{1-p}^{(1-k^{-1})\bink n2-m+\cH_G(\sigma)}\\
		&\hspace{-4cm}=
				\Theta(1)\bink{\bink n2}{m}^{-1}\sum_{(G,\sigma)\in\cA,\sigma\in\cB,|E(G)|=m}\exp(-\beta\cH_G(\sigma))
				=\Theta(1)\sum_{\sigma\in\cB}\Erw\brk{\exp(-\beta\cH_{\G}(\sigma))\vecone_{\{(\G,\sigma)\in\cA\}}}.
	\end{align*}
Thus, the assertion follows from (\ref{eqcontig1}).
\end{proof}

\noindent
We are going to combine \Lem~\ref{lem:contig} with the following proposition, which shows that separability is a likely event in the planted model.

\begin{proposition}	\label{prop:sepbal}
We have $\pr[\hat\SIGMA\text{ is separable in }\hat\G|\hat\SIGMA\in\cB]=1-o(n^{-1/2})$.
\end{proposition}

To prove Proposition \ref{prop:sepbal} we generalize the argument for proper $k$-colorings from \cite[Section 3]{Danny} to the Potts antiferromagnet.
In the following we let $V_i=\hat\SIGMA^{-1}(i)$ for $i\in[k]$.

\begin{lemma}	\label{lem:L1}
Let $i\in[k]$.
For $S\subset V_i$ let $X_{S,i}=|\cbc{v\in V\setminus V_i:\partial_{\hat\G} v\cap S = \emptyset}|$.
Given $\hat\SIGMA\in\cB$  the following statement holds with probability $1-\exp(-\Omega(n))$.
	\begin{equation}\label{eqPM1}
	\parbox{14cm}{Let $i\in[k]$.
	Then for all $S\subset V_i$ of size $\frac kn|S|\in[0,501,1-k^{-0.499}]$
	we have $X_{S,i}\leq\frac{n}{k}(1-\alpha-\kappa)-n^{2/3}$.}
	\end{equation}
\end{lemma}
\begin{proof}
It suffices to prove the statement for $i=1$ and we set $X_S=X_{S,1}$.
Moreover, let $\alpha\in[0.501,1-k^{-0.499}]$.
For a fixed $S\subset V_1$ and $v\in V\setminus V_1$ the number $|\partial_{\hat\G} v\cap S|$ is a binomial random variable
	with parameters $|S|=\frac{\alpha n}k$ and $p_2$.
Hence, $\pr[\partial v\cap S=\emptyset]=(1-p_2)^{|S|}$.
Consequently, $X_S$ itself is a binomial variable with mean $|V\setminus V_1|(1-p_2)^{|S|}$. Because $\sigma$ is balanced, we have
	 $|V\setminus V_1|\sim n(1-1/k)$.
Further, our assumptions on $d,\beta$ entail
	\begin{align*}
	(1-p_2)^{|S|}\le\exp(-p_2|S|)\le\exp\left(-\alpha\frac{2k\ln k}{k-c_\beta}+\alpha \frac{3\ln k}{k-c_\beta}\right)\le(1+o_k(1))k^{-2\alpha}.
	\end{align*}
Therefore, $\Erw[X_S]\le n(1+o_k(1))(1-1/k)k^{-2\alpha}$.
Thus, Lemma \ref{lem:chernoff} yields
	\[\pr[X_S>(1-\alpha-\kappa)\frac{n}{k}-n^{2/3}]\le
		\exp\left[-(1-\alpha-\kappa+o(1))\frac{n}{k}\ln\left(\frac{1-\alpha-\kappa}{ek^{1-2\alpha}}\right)\right].\]
The total number of sets $S$ of size $\alpha n/k$ is
	$\binom{|V_1|}{\alpha \frac{n}{k}}
		\leq\exp\left[
	\frac{n}{k}h(\alpha)\right].$
Hence, by the union bound
	\begin{align}\nonumber
	\pr\left[\exists S:X_S\ge(1-\alpha-\kappa)\frac{n}{k}-n^{2/3}\right]
	&\le  
	\exp\left[\frac nk\bc{2h(\alpha)+(1-\alpha)+(1-2\alpha)(1-\alpha-\kappa)\ln k
		+o(1)}\right]\\
	&\hspace{-2cm}\leq\exp\left[\frac nk\bc{(1-\alpha)(3-2\ln(1-\alpha))+(2(1-\alpha)^2-(1-2\kappa)(1-\alpha)+\kappa)\ln k
		+o(1)}\right].
	\label{eq:unionS}
	\end{align}
Substituting $y=1-\alpha$ and differentiating, we obtain
	\begin{align*}
	\frac{\partial}{\partial y}y(3-2\ln y)+(2y^2-(1-2\kappa)y+\kappa)\ln k&=1-2\ln y+4y\ln k-(1-2\kappa)\ln k,\\
	\frac{\partial^2}{\partial y^2}y(3-2\ln y)+(2y^2-(1-2\kappa)y+\kappa)\ln k&=
		-\frac2y+4\ln k,\qquad
	\frac{\partial^3}{\partial y^3}y(3-2\ln y)+(2y^2-(1-2\kappa)y+\kappa)\ln k=2.
		\end{align*}
Hence, the first derivative is negative at the left boundary point $y=k^{-0.499}$, positive at the right boundary point $y=0.499$
and convex on the entire interval.
Furthermore, we check that
$y(3-2\ln y)+(2y^2-(1-2\kappa)y+\kappa)\ln k<0$ for $y\in\{0.499,k^{-0.499}\}$.
Therefore, the assertion follows from (\ref{eq:unionS}).
\end{proof}

\begin{lemma}	\label{lem:L2}
Given $\hat\SIGMA\in\cB$ the random graph $\hat\G$ has the following property with probability $1-\exp(-\Omega(n))$.
	\begin{equation}\label{eqPM2}
	\parbox{14cm}{Let $i\in[k]$ and let $Y=Y(\hat\G,\hat\SIGMA)$ be the number of vertices $v\not\in V_i$ with fewer than 15 neighbors in $V_i$.
		Then $Y\leq\frac{\kappa n}{3k\ln k}$.}
	\end{equation}
\end{lemma}
\begin{proof}
Suppose $i=1$.
Given $\hat\SIGMA\in\cB$ for $v\notin V_1$ the number $|\partial_{\hat\G}v\cap V_1|$ of neighbors in $V_1$ is a binomial variable with mean
	$\lambda=|V_1|p_2\sim {d}/\bc{k-c_\beta}>2\ln k+O_k(\ln k/k).$
Hence, the probability of a vertex having at most 14 neighbors in $V_1$ is upper bounded by $2\lambda^{14}\exp(-\lambda)\le 3k^{-2}\ln^{14}k$.
Therefore, $Y$ is dominated by a binomial variable with mean $\mu\le 3nk^{-2}\ln^{14}k$.
Finally, the assertion follows from Lemma \ref{lem:chernoff} and the choice of $\kappa$.
\end{proof}

\begin{claim}\label{claim:smalledgespan}
Given $\hat\SIGMA\in\cB$ the random graph $\hat\G$ has the following property with probability $1-O(n^{-1})$.
	\begin{equation}\label{eqPM3}
	\parbox{14cm}{If $W\subset V$ has size $W\le k^{-4/3}n$, then  $W$ spans no more than $5|W|$ edges.}
	\end{equation}
\begin{proof}
Given $\hat\SIGMA$ for any edge of the complete graph the probability of being present in $\hat\G$ is bounded by $p_2$.
Therefore, by the union bound and with room to spare, for any $0<\gamma\le k^{-4/3}$ we find
		\begin{align*}
		&\pr\left[\exists W\subset V,\,|W|=\gamma n:W\mbox{ spans $5|W|$ edges}\right|\hat\SIGMA\in\cB]
			\le \bink{n}{\gamma n}\binom{\binom{\gamma n}{2}}{5\gamma n}p_2^{5\gamma n}
		\le 
			\brk{\frac{e}{\gamma}\bcfr{e\gamma d}{5}^5}^{\gamma n}\leq\bc{\gamma^4d^5}^{\gamma n}.
		\end{align*}
Summing over $1/n\leq\gamma\leq k^{-4/3}$ completes the proof.
	\end{proof}
\end{claim}

\begin{proof}[Proof of Proposition \ref{prop:sepbal}]
Suppose that $\hat\SIGMA$ is balanced.
By our assumptions on $d,\beta$ for each $i$ the number of edges spanned by $\hat\SIGMA^{-1}(i)$ in $\hat\G$ is a binomial random variable with mean
	$$(1+o(1))\bink{n/k}2p_1\leq (1+o(1))\frac{dn}{2k(k-c_\beta)}\exp(-\beta)\leq(1+o_k(1))nk^{-1}\exp(-\beta)\ln k.$$
Hence, \Lem~\ref{lem:chernoff} shows that $(\hat\G,\hat\SIGMA)$ satisfies {\bf SEP1} with probability $1-\exp(-\Omega(n))$.

With respect to {\bf SEP2}, we continue to condition on $\hat\SIGMA\in\cB$.
By \Lem~\ref{lem:L1}, \Lem~\ref{lem:L2} and Claim~\ref{claim:smalledgespan}
	we may assume that $\hat\G$ has the properties (\ref{eqPM1}), (\ref{eqPM2}) and (\ref{eqPM3}).
In order to show separability we may without loss of generality restrict ourselves to the case of $i=j=1$.
Thus, suppose that $\tau\in\Sigma_{\G,\beta}$ satisfies $\rho_{11}(\hat\SIGMA,\tau)\ge0.51\frac{n}{k}$
and assume for contradiction that $\alpha=\frac kn|S|=\rho_{11}(\hat\SIGMA,\tau)<1-\kappa$.
Let $$S=\hat\SIGMA^{-1}(1)\cap \tau^{-1}(1),\quad R=\hat\SIGMA^{-1}(1)\setminus\tau^{-1}(1),\quad T=\tau^{-1}(1)\setminus\hat\SIGMA^{-1}(1).$$
Because $\sigma$ and $\tau$ are balanced, we have
	\begin{align}\label{eqContra1}
	|T\cup S|\sim\frac{n}{k}\sim|R\cup S|.
	\end{align}
Let $T_0=\{v\in T:\partial_{\hat\G}v\cap S=\emptyset\}$ and let $T_1=T\setminus T_0$.
Then \textbf{SEP1} and our assumptions on $d$ and $\beta$ ensure that
	$|T_1|\le \frac{4n\ln k}{k\exp(\beta)}.$
Consequently, the assumption $\beta\geq\ln k$ yields
	\begin{align*}
	|T_0|\ge \frac{n}{k}(1-\alpha- O_k(\ln k/k)).
	\end{align*}
Since the vertices in $T_0$ do not have neighbors in $S$, (\ref{eqPM1}) implies that 
	\begin{align}\label{eqContra3}
	\alpha>1-k^{-0.49}.
	\end{align}
Further, let
	$U=\{v\in T\,:\,|\partial v \cap \hat\SIGMA^{-1}(1)|\ge 15\}.$
Then (\ref{eqPM2}) implies that $|T|\leq|U|+\kappa n/(k\ln k)$.
Therefore, (\ref{eqContra1}) and our assumption $\alpha<1-\kappa$ yield
	\begin{align}\label{eqContra2}
	|U|\geq(1-o_k(1))\frac{\kappa n}{k}\qquad\mbox{and}\qquad
	|R|-o_k(\kappa)\frac nk\leq|U|\leq|R|+o(n).
	\end{align}
Hence, {\bf SEP1} implies that $S\cup U$ spans no more than $2nk^{-1}\exp(-\beta)\ln k\leq|U|$ edges.
Consequently, $U\cup R$ spans at least $14|U|$ edges.
Thus, combining (\ref{eqPM3}) and (\ref{eqContra2}), we conclude that $|U\cup R|> nk^{-4/3}$.
But then (\ref{eqContra1}) and (\ref{eqContra2}) show that $1-\alpha+o(1)\geq\frac kn|R|>\frac k{3n}|U\cup R|\geq\frac13k^{-1/3}$,
in contradiction to (\ref{eqContra3}).
\end{proof}

\begin{proof}[Proof of Proposition \ref{prop:expsepbal}]
	By linearity of expectation, applying Lemma \ref{lem:contig} to Proposition \ref{prop:sepbal} yields $\Erw[\Zkb{\G}]\sim \Erw[\Zkg{\G}]$. 
\end{proof}

\section{High degree, low temperature: the second moment}\label{sec:doublystoch}

\noindent
To prove \Prop~\ref{prop:smm} we call a doubly-stochastic $k\times k$-matrix $\rho$ {\em separable} if $\rho_{ij}\not\in(0.51,1-\kappa)$ for all $i,j\in[k]$.
Moreover, $\rho$ is \emph{$s$-stable} if $s=|\{(i,j)\in[k]^2:\rho_{ij}>0.51\}|$.
Let $\Dg\subset\cD$ be the set of all separable matrices and let $\Dsg\subset\Dg$ be the set of all $s$-stable matrices so that $\Dg=\bigcup_{s=0}^k\Dsg$.
The key step is to optimize the function $f_{d,\beta}$ over $\Dg$.

\begin{proposition}\label{prop:maxsep}
If $2(k-1)\ln(k-1)\leq d\leq d_\star$ and $\beta\geq\ln k$, then
$f_{d,\beta}(\rho)<f_{d,\beta}(\bar{\rho})$
for all $\rho\in\Dg\setminus\{\bar{\rho}\}$.
\end{proposition}

A similar statement for the function
	\begin{align}
	f_{d,\infty}(\rho)=H(k^{-1}\rho)+\frac{d}{2}\ln\left[1-\frac{2}{k}+\frac{\|\rho\|_2^2}{k^2}\right],\label{eq:frozenF}
	\end{align} 
the limit of $f_{d,\beta}(\rho)$ as $\beta\to\infty$,
played a key role in~\cite{Danny}.
Specifically, we have

\begin{proposition}[{\cite[Propositions 4.4--4.6, 4.8]{Danny}}]\label{Prop_Danny}
	Assume that $d=(2k-1)\ln k - 2\ln 2$.
Then $f_{d,\infty}(\rho)<f_{d,\infty}(\bar{\rho})$ for all $0\le s < k$, $\rho\in \Dng{s}\setminus\{\bar{\rho}\}$.
\end{proposition}		

We prove \Prop~\ref{prop:maxsep} by combining \Prop~\ref{Prop_Danny} with monotonicity in both $d$ and $\beta$.
In fact, \Lem~\ref{lem:monotony} readily provided monotonicity in $\beta$.
Further, with respect to $d$ we have the following.

\begin{lemma}
	\label{lem:monotonyd}
	For every $d>0$, $\rho\in\cS$ we have
		\begin{align*}
		\frac{\partial}{\partial d}f_{d,\infty}(\bar\rho)&\leq\frac{\partial}{\partial d}f_{d,\infty}(\rho)<0.
		\end{align*}
Hence,  if $f_{d',\beta}(\bar{\rho})\ge f_{d',\beta}(\rho)$, then $f_{d,\beta}(\bar{\rho})\ge f_{d,\beta}(\rho)$ for all $0\leq d<d'$.
	\begin{proof}
Recalling that $1\leq \|\rho\|_2^2\leq k$, we find
		\begin{align*}
		\frac{\partial}{\partial d}f_{d,\infty}(\rho)=\frac{1}{2}\ln \left(1-\frac{2}{k}+\frac{\|\rho\|_2^2}{k^2}\right)<0.
		\end{align*}
The assertion follows because $\bar\rho$ minimizes the Frobenius norm on $\cS$.
	\end{proof}
\end{lemma}

\begin{corollary}
	   \label{cor:MaxCases}
		Let $\beta\ge 0$ and $d\le d_{\star}$. For all $1\le s\le k-1$ and $\rho\in\bigcup_{s<k}\Dng{s}\setminus\{\bar\rho\}$
			we have $f_{d,\beta}(\rho)<f_{d,\beta}(\bar{\rho})$.
\end{corollary}
\begin{proof}
Suppose that $\rho\in\bigcup_{s<k}\Dng{s}\setminus\{\bar\rho\}$.
Combining \Prop~\ref{Prop_Danny} with \Lem~\ref{lem:monotonyd}, we see that
	\begin{align*}
	\lim_{\gamma\to\infty}f_{d,\gamma}(\rho)=f_{d,\infty}(\rho)< f_{d,\infty}(\bar\rho)=\lim_{\gamma\to\infty}f_{d,\gamma}(\bar\rho).
	\end{align*}
Hence, $f_{d,\gamma}(\rho)<f_{d,\gamma}(\bar\rho)$ for $\gamma>\beta$ sufficiently large.
Therefore, \Lem~\ref{lem:monotony} entails that $f_{d,\beta}(\rho)<f_{d,\beta}(\bar\rho)$.
\end{proof}

\noindent
Observe that \Prop~\ref{Prop_Danny} (and hence \Cor~\ref{cor:MaxCases}) does not cover the $k$-stable case.
			
			\begin{lemma}
				\label{lem:MaxCases4}
				Let $\beta\ge 0$ and $d\le d_{\star}$. For all $\rho\in\Dng{k}$ we have $f_{d,\beta}(\rho)<f_{d,\beta}(\bar{\rho})$.
			\end{lemma}
			\begin{proof}
				Because $\rho\in\Dng{k}$ for each $i\in[k]$ there is precisely one entry greater than 0.51. Without loss of generality we may assume that $\rho_{ii}\ge0.51$. By Lemma \ref{lem:FormMaxMatrix} there is $\alpha=k^{-1}+\tilde{O}_k(k^{-2})$ such that the matrix $\rho'$ obtained from $\rho$ by substituting any row $\rho_i$ for a row $\rho_i'$ with $\rho'_{ii}=1-\alpha$ and $\rho'_{ij}=\alpha/(k-1)$ for $j\neq i$ satisfies $f_{d,\beta}(\rho')\ge f_{d,\beta}(\rho)$. Hence, a maximizer of $\rho\in\Dng{k}$ is of the form $\rhos=(1-1/k)\id + 1/k^2\vecone$. Because the matrix $\rhos$ does not further improve from applying the transformation in Lemma \ref{lem:FormMaxMatrix}, it remains to show that 
				\begin{align}
					f_{d,\beta}(\bar{\rho})>f_{d,\beta}(\rhos).\label{eq:finalineq}
				\end{align}
				In the zero temperature case with $d=(2k-1)\ln k - c$, we have
				\begin{align}
					f_{d,\infty}(\bar{\rho})&=\ln k +\frac{1}{k}\sum_{i\le k}H(\bar{\rho}_i)+\frac{d}{2}\ln\left[1-\frac{2}{k}1+\|\bar{\rho}\|_2^2\frac{1}{k^2}\right]=2\ln k + d\ln\left[1-\frac{1}{k}\right]\tag*{[as $\|\bar{\rho}\|_2^2=1$]}\nonumber\\
					&=2\ln k -d\left(\frac{1}{k}+\frac{1}{2k^2}+O_k\left({k^{-3}}\right)\right) = 2\ln k -(2k\ln k - \ln k -c)\left(\frac{1}{k}+\frac{1}{2k^2}+O_k\left({k^{-3}}\right)\right)\nonumber\\
					&=\frac{c}{k}+O_k\left(\frac{\ln k}{k^2}\right).\label{eq:fbarrhofin}
				\end{align}
				On the other hand the matrix $\rhos$ satisfies
				\begin{align}
					H(k^{-1}\rhos)&=\ln k + \frac{1}{k}\sum_{i\le k}H\left((1-1/k+1/k^2,1/k^2,\ldots,1/k^2)\right)\nonumber\\
					&=\ln k - \left(1-\frac{1}{k}+\frac{1}{k^2}\right)\ln\left(1-\frac{1}{k}+\frac{1}{k^2}\right)+\frac{(k-1)}{k^2}\ln k^2.\label{eq:entropyrhos}
				\end{align}
				Because $\|\rhos\|_2^2=\frac{k(k-1)}{k^4}+k(1-\frac{1}{k}+\frac{1}{k^2})^2$ and $\beta\ge \ln k$, setting $d=(2k-1)\ln k - c$ we obtain
				\begin{align}
					E(\rhos)&=\frac{d}{2}\ln\left[1-\frac{2}{k}+\frac{1}{k^2}\left(\frac{k(k-1)}{k^4}+k\left(1-\frac{1}{k}+\frac{1}{k^2}\right)^2\right)\right]\nonumber\\
					&=\frac{d}{2}\ln\left[1-\left(\frac{1}{k}+\frac{2}{k^2}+O_k\left({k^{-3}}\right)\right)\right]=-\frac{d}{2}\left(\frac{1}{k}+\frac{1}{k^2}+\frac{1}{2}\left(\frac{1}{k}+\frac{2}{k^2}\right)^2+O_k\left(k^{-3}\right)\right)\nonumber\\
					&=-\left(k\ln k -\frac{\ln k}{2}-\frac{c}{2}\right)\left(\frac{1}{k}+\frac{5}{2k^2}+O_k\left(k^{-3}\right)\right)\nonumber\\
					&=-\ln k -\frac{2\ln k}{k}+\frac{c}{2k}+O_k\left(\frac{\ln k}{k^2}\right).\label{eq:energyrhos}
				\end{align}
				Consequently
				\begin{align}
					f_{d,\infty}(\rhos)&=\frac{1}{k}+\frac{c}{2k}+O_k\left(\frac{\ln k}{k^2}\right).\label{eq:frhos}
				\end{align}
				From \eqref{eq:fbarrhofin} and \eqref{eq:frhos} we see that $f_{d,\beta}(\bar{\rho})>f_{d,\beta}(\rhos)$ holds for any $d\le (2k-1)\ln k - 2 -\omega_k(\ln k/k)$ and $\beta=\infty$. Lemma \ref{lem:monotony} concludes the proof by extending \eqref{eq:finalineq} to $\beta\ge \ln k$.
			\end{proof}

\begin{proof}[Proof of \Prop~\ref{prop:maxsep}]
Because $\Dg$ decomposes into disjoint subsets $\Dsg$, $s=0,1,\ldots,k$ Proposition \ref{prop:maxsep} is immediate from Corollary \ref{cor:MaxCases} and Lemma \ref{lem:MaxCases4}.
\end{proof}

			\begin{proof}[Proof of Proposition \ref{prop:smm}]
				By definition of $\cB_{sep}$ and Proposition \ref{prop:sepbal} we have
				\begin{align}
					\Erw[\Zkg{\G}^2]\sim\sum_{\rho\in\cR\cap\cB_{sep}}\Erw[\Zrb(\G)].\label{eq:smm1}
				\end{align}
By Propositions \ref{prop:exppairs} and~\ref{prop:maxsep},
				\begin{align}
					\sum_{\rho\in\cR\cap\cB_{sep}}\Erw[\Zrb(\G)]=\sum_{\rho\in\cR\cap\cB_{sep}}\exp(nf_{d,\beta}(\rho)+o(n))	=\exp\left(2n\ln k+nd\ln\left(1-c_\beta/k\right)+o(n)\right).\label{eq:smm3}
				\end{align}
	Combining \eqref{eq:smm1}--\eqref{eq:smm3} with \eqref{eq:balO} and taking logarithms yields the assertion.
			\end{proof}


\begin{thebibliography}{29}
	
	\bibitem{AchFried}
	D.~Achlioptas, E.~Friedgut:
	A sharp threshold for $k$-colorability.
	Random Struct.\ Algorithms {\bf 14} (1999) 63--70.
	
	\bibitem{Barriers}
	D.~Achlioptas, A.~Coja-Oghlan:
	Algorithmic barriers from phase transitions.
	Proc.~49th FOCS (2008) 793--802.
	
	\bibitem{AchMolloy}
	D.\ Achlioptas, M.\ Molloy:
	The analysis of a list-coloring algorithm on a random graph.
	Proc.\ 38th FOCS (1997) 204--212.
	
	\bibitem{AchNaor}
	D.~Achlioptas, A.~Naor:
	The two possible values of the chromatic number of a random graph.
	Annals of Mathematics {\bf 162} (2005), 1333--1349.
	
	\bibitem{PeterCatherine}
	P.\ Ayre, A.\ Coja-Oghlan, C.\ Greenhill:
	Hypergraph coloring up to condensation.
	arXiv:1508.01841 (2015).
	
	\bibitem{Banks}
	J.\ Banks, C.\ Moore:
	Information-theoretic thresholds for community detection in sparse networks.
	arXiv:1601.02658 (2016).
	
	\bibitem{Victor}
	V.~Bapst, A.~Coja-Oghlan:
	The condensation phase transition in the regular $k$-SAT model.
		arXiv:1507.03512 (2015).

	\bibitem{Cond}
	V.~Bapst, A.~Coja-Oghlan, S.~Hetterich, F.~Ra\ss mann, Dan Vilenchik:
	The condensation phase transition in random graph coloring.
	Communications in Mathematical Physics {\bf 341} (2016) 543--606.

	\bibitem{h2c}
	V.~Bapst, A.~Coja-Oghlan, F.~Ra\ss mann:
	A positive temperature phase transition in random hypergraph $2$-coloring.
	Annals of Applied Probability, in press.

	
	\bibitem{Baxter}
	R.~Baxter: 
	Exactly solved models in statistical mechanics. 
	Courier Corporation, (2007).
	

	\bibitem{BST}
	N.\ Bhatnagar, A.\ Sly, P.\ Tetali:
	Decay of correlations for the hardcore model on the $d$-regular random graph.
	arXiv:1405.6160 (2014).
	
	\bibitem{Covers}
	A.~Coja-Oghlan: Upper-bounding the $k$-colorability threshold by counting covers.
	Electronic Journal of Combinatorics {\bf 20} (2013) P32.
	
	\bibitem{Reg}
	A.~Coja-Oghlan, C.\ Efthymoiu, S.\ Hetterich: On the chromatic number of random regular graphs.
	Journal of Combinatorial Theory, Series B {\bf116}  (2016) 367--439.
	
	\bibitem{Danny}
	A.~Coja-Oghlan, D.~Vilenchik:
	The chromatic number of random graphs for most average degrees.
	International Mathematics Research Notices, in press.
	
	\bibitem{CDGS}
	P.~Contucci, S.~Dommers, C.~Giardina, S.~Starr:
	Antiferromagnetic Potts model on the \Erdos-\Renyi\ random graph.
	Communications in Mathematical Physics {\bf  323} (2013) 517--554.
		
	\bibitem{DMSS}
	A.\ Dembo, A.\ Montanari, A.\ Sly, N.\ Sun: The replica symmetric solution for Potts models on $d$-regular graphs.
	Comm.\ Math.\ Phys.\ {\bf 327} (2014) 551--575.

	\bibitem{DFG}
	M.~Dyer, A.~Frieze, C.~Greenhill:
	On the chromatic number of a random hypergraph.
	Journal of Combinatorial Theory, Series B.
	\textbf{113} (2015), 68-122
	
	\bibitem{Efthy14}
	C.~Efthymiou: MCMC sampling colourings and independent sets of G(n, d/n) near uniqueness threshold. 
	Proc.\ 25th SODA (2014) 305--316.
	
	\bibitem{Efthy12}
	C.~ Efthymiou: Switching colouring of $G(n, d/n)$ for sampling up to Gibbs uniqueness threshold. 
	Proc.\ 22nd ESA (2014)  371-381
		
	\bibitem{ER}
	P.\ Erd\H os, A.\ R\'enyi: On the evolution of random graphs.
	Magayar Tud.\ Akad.\ Mat.\ Kutato Int.\ Kozl.\ {\bf 5} (1960) 17--61.
	
	\bibitem{janson2011random}
	S.~Janson, T.~Luczak, A.~Rucinski:
	Random graphs. Vol. 45. John Wiley \& Sons, (2011).
		
	\bibitem{KPGW}
	G.~Kemkes, X.~P\'erez-Gim\'enez, N.~Wormald:
	On the chromatic number of random $d$-regular graphs.
	Advances in Mathematics {\bf 223}  (2010) 300--328.
		
	\bibitem{pnas}
	F.~Krzakala, A.~Montanari, F.~Ricci-Tersenghi, G.~Semerjian, L.~Zdeborova:
	Gibbs states and the set of solutions of random constraint satisfaction problems.
	Proc.~National Academy of Sciences {\bf104} (2007) 10318--10323.
		
	\bibitem{MM}
	M.~M\'ezard, A.~Montanari:
	Information, physics and computation.
	Oxford University Press~2009.
	
	\bibitem{MMtrees}
	M.~M\'ezard, A.~Montanari:
	Reconstruction on trees and spin glass transition.
	Journal of statistical physics 124.6 (2006): 1317-1350.
		
	\bibitem{Molloy}
	M.~Molloy: The freezing threshold for $k$-colourings of a random graph.
	Proc.\ 43rd STOC (2012) 921--930.
	

	\bibitem{LFglass}
	L.~Zdeborov\'a, F.~Krzakala: 
	Potts Glass on Random Graphs.
	EPL (Europhysics Letters) 81.5 (2008): 57005.
	
	\bibitem{LenkaFlorent}
	L.~Zdeborov\'a, F.~Krzakala: Phase transition in the coloring of random graphs.
	Phys.\ Rev.\ E {\bf76} (2007) 031131.

	
\end{thebibliography}
\end{document}